\documentclass[a4paper,11pt]{article}

\usepackage[T1]{fontenc}
\usepackage{lmodern}

\usepackage{dsfont}

\usepackage[latin1]{inputenc}
\usepackage{amsmath}
\usepackage{amsthm}
\usepackage{amssymb}
\usepackage{mathrsfs}
\usepackage{graphicx}
\usepackage[all]{xy}
\usepackage{hyperref}

\usepackage{makeidx}

\usepackage{stmaryrd}
\usepackage{caption}

\usepackage{abstract} 

\newtheorem{thm}{Theorem}[section]
\newtheorem{cor}[thm]{Corollary}
\newtheorem{claim}[thm]{Claim}
\newtheorem{fact}[thm]{Fact}

\newtheorem{lemma}[thm]{Lemma}
\newtheorem{prop}[thm]{Proposition}

\theoremstyle{definition}
\newtheorem{definition}[thm]{Definition}

\newtheorem{question}[thm]{Question}

\newtheorem{conj}[thm]{Conjecture}

\def\rquotient#1#2{%
	\makeatletter
	\raise.3ex\hbox{$#1$}/\lower.3ex\hbox{$#2$}%
	\makeatother
}	

\makeatletter
\newcommand{\subjclass}[2][2010]{%
	\let\@oldtitle\@title%
	\gdef\@title{\@oldtitle\footnotetext{#1 \emph{Mathematics subject classification.} #2}}%
}
\newcommand{\keywords}[1]{%
	\let\@@oldtitle\@title%
	\gdef\@title{\@@oldtitle\footnotetext{\emph{Key words and phrases.} #1.}}%
}
\makeatother

\newcommand{\Address}{{
		\bigskip
		\small
		
\noindent \textsc{Universit\'e de Montpellier\\ 
Institut Montpellierain Alexander Grothendieck\\
Place Eug\`ene Bataillon\\
34090 Montpellier (France)}\par\nopagebreak
\noindent \textit{E-mail address}: \texttt{anthony.genevois@umontpellier.fr}
		
}}

\makeindex

\title{Flat braid groups, right-angled Artin groups, and commensurability}
\date{\today}
\author{Anthony Genevois}


\begin{document}

\maketitle

\begin{abstract}
For every $n\geq 1$, the flat braid group $\mathrm{FB}_n$ is an analogue of the braid group $B_n$ that can be described as the fundamental group of the configuration space
$$\left\{ \{x_1, \ldots, x_n \} \in \mathbb{R}^n / \mathrm{Sym}(n) \mid \text{there exist at most two indices $i,j$ such that } x_i=x_j \right\}.$$
Alternatively, $\mathrm{FB}_n$ can also be described as the right-angled Coxeter group $C(P_{n-2}^\mathrm{opp})$, where $P_{n-2}^\mathrm{opp}$ denotes the opposite graph of the path $P_{n-2}$ of length $n-2$. In this article, we prove that, for every $n= 7$ or $\geq 11$, $\mathrm{PFB}_n$ is not virtually a right-angled Artin group, disproving a conjecture of Naik, Nanda, and Singh. In the opposite direction, we observe that $\mathrm{FB}_7$ turns out to be commensurable to the right-angled Artin group $A(P_4)$.
\end{abstract}

\tableofcontents

\newpage
\section{Introduction}

\noindent
Recall that, given a graph $\Gamma$, the corresponding \emph{right-angled Artin group} is
$$A(\Gamma):= \langle u \text{ vertex of } \Gamma \mid [u,v]=1 \text{ if $u$ and $v$ are adjacent in } \Gamma \rangle,$$
and that the corresponding \emph{right-angled Coxeter group} is
$$C(\Gamma):= \langle u \text{ vertex of } \Gamma \mid u^2=1 \text{ for every $u$, } [u,v]=1 \text{ if $u$ and $v$ are adjacent in } \Gamma \rangle.$$
It is well-known that right-angled Artin and Coxeter groups are tightly connected. Most notably, every right-angled Artin group turns out to be isomorphic to a finite-index subgroup of some right-angled Coxeter group \cite{MR1783167}. Conversely, despite the fact that every right-angled Coxeter group can be virtually described as a subgroup of some right-angled Artin group, right-angled Coxeter groups may not be commensurable to right-angled Artin groups. A simple example is given by $C(C_n)$, where $C_n$ is a cycle of length $n \geq 5$. Indeed, $C(C_n)$ can be described as the reflection group associated to a right-angled $n$-gon in the hyperbolic plane $\mathbb{H}^2$, so $C(C_n)$ is virtually the fundamental group of a closed surface of genus $\geq 2$. However, a right-angled Artin group that does not admit $\mathbb{Z}^2$ as a subgroup is automatically free, so no right-angled Artin group can be commensurable to $C(C_n)$. Consequently, one can think of the family of right-angled Coxeter groups as being strictly larger than the family of right-angled Artin groups. A natural problem, then, is to understand when a group from the bigger family belongs to the smaller family. More precisely:

\begin{question}\label{question:MainQuestion}
Given a graph $\Gamma$, when is the right-angled Coxeter group $C(\Gamma)$ commensurable to a right-angled Artin group?
\end{question}

\noindent
Recall that two groups are (\emph{abstractly}) \emph{commensurable} whenever they contain two isomorphic finite-index subgroups. 

\medskip \noindent
Question~\ref{question:MainQuestion}, and its analogue where ``commensurable'' is replaced with ``quasi-isometric'', is well-known in geometric group theory. Nevertheless, it is poorly understood. Some invariants are available in order to distinguish some right-angled Artin and Coxeter groups, such as divergence and thickness \cite{MR2874959, MR3314816, MR4584587} or Morse boundaries \cite{MR3339446, MR3664526, CounterBehr} and Morse subgroups \cite{MR4464462}. In the other directions, a few constructions are known in order to produce (finite-index) subgroups in right-angled Coxeter groups that are right-angled Artin groups. See for instance \cite{MR1875609}, and most notably \cite{MR4735045} (based on \cite{MR3768473} and further studied in \cite{Visual}). As a concrete but very specific application, it is determined in \cite{MR4735045} precisely when some two-dimensional one-ended right-angled Coxeter groups defined by planar graphs are commensurable to right-angled Artin groups. (These examples are not representative of the general case because they are based on the large-scale geometry of graph manifolds \cite{MR2376814, MR4023336}.) Despite all these results available in the literature, no global picture seems to emerge and an answer to Question~\ref{question:MainQuestion} seems currently to be out of reach in full generality. 

\medskip \noindent
In this article, we focus on a specific family of right-angled Coxeter groups, known as \emph{flat braid groups}. For every $n \geq 1$, the flat braid group $\mathrm{FB}_n$  on $n$ strands is the fundamental group of the configuration space
$$\left\{ \{x_1, \ldots, x_n \} \in \mathbb{R}^n / \mathrm{Sym}(n) \mid \text{there exist at most two indices $i,j$ such that } x_i=x_j \right\}.$$
One can think of an element of $\mathrm{FB}_n$ as a configuration of $n$ arcs in the infinite strip $\mathbb{R} \times [0,1]$ connecting $n$ marked points on each of the parallel lines $\mathbb{R} \times \{1\}$ and $\mathbb{R} \times \{ 0\}$ such that each arc is monotonic and no three arcs intersect at a common point. Two such configurations are considered as equivalent if one can be deformed into the other by a homotopy of such configurations in $\mathbb{R} \times [0,1]$ keeping the end points of the arcs fixed. See Figure~\ref{FlatBraid}. From this description, one can define a natural morphism $\mathrm{FB}_n \to \mathrm{Sym}(n)$, encoding how a flat braid permutes the $n$ strands. The kernel of this morphism is the \emph{pure flat braid group} $\mathrm{PFB}_n$.
\begin{figure}
\begin{center}
\includegraphics[width=0.6\linewidth]{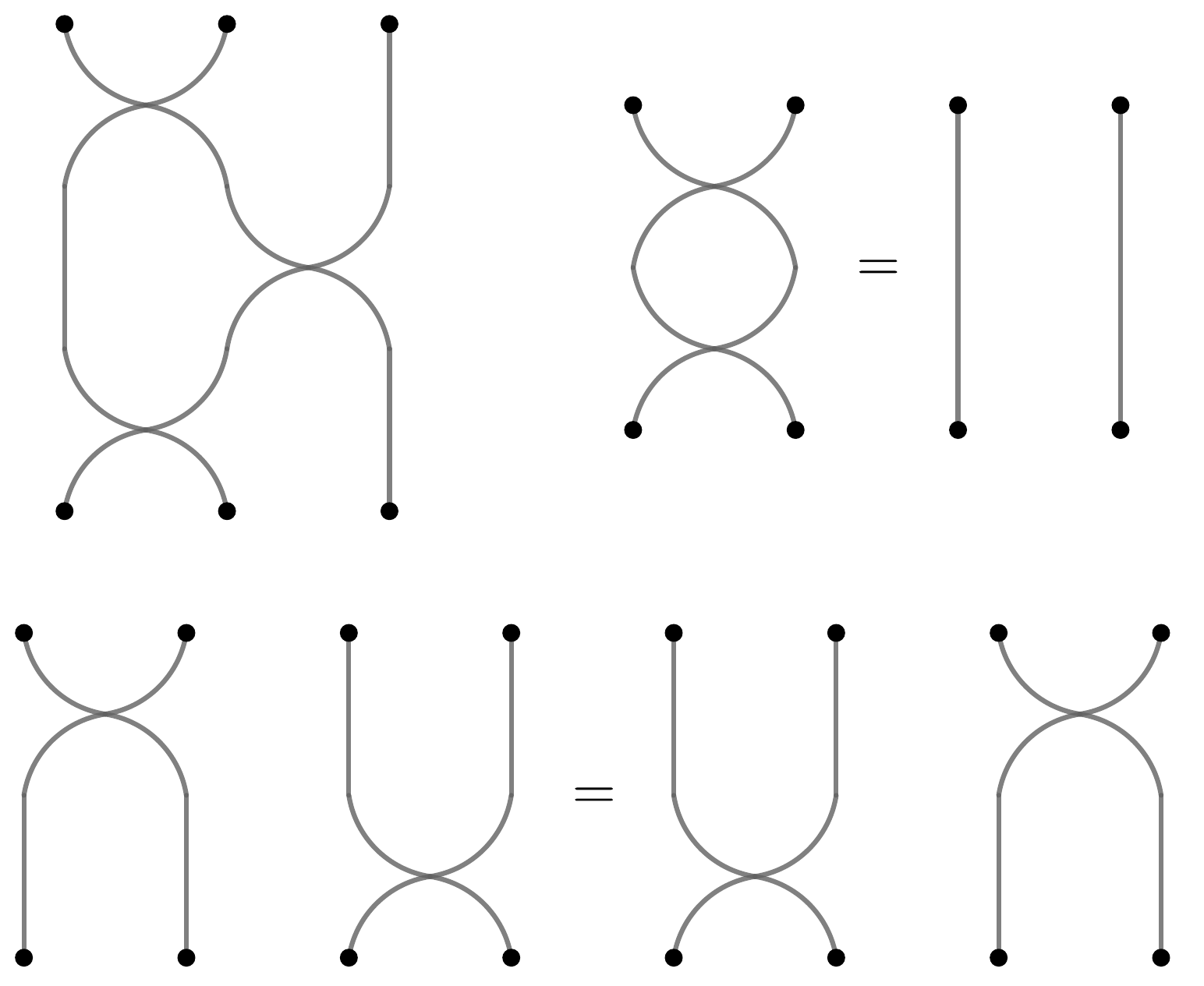}
\caption{A flat braid from $\mathrm{FB}_3$ (namely, $\sigma_1 \sigma_2 \sigma_1$), and the two typical relations between flat braids (namely, $\sigma_1^2=1$ and $\sigma_1\sigma_3= \sigma_3 \sigma_1$.}
\label{FlatBraid}
\end{center}
\end{figure}

\medskip \noindent
If $\sigma_i$ denotes the element of $\mathrm{FB}_n$ that twists the $i$th and $(i+1)$st strands, then it is clear that $\mathrm{FB}_n$ is generated by $\{\sigma_1, \ldots, \sigma_{n-1}\}$. Moreover, $\mathrm{FB}_n$ admits
$$\langle \sigma_1 ,\ldots, \sigma_{n-1} \mid \sigma_i^2=1 \text{ for every } 1 \leq i \leq n-1, \ [\sigma_i,\sigma_j]=1 \text{ whenever } |i-j| \geq 2 \rangle$$
as a presentation. Thus, $\mathrm{FB}_n$ coincides with the right-angled Coxeter group $C(P_{n-2}^\mathrm{opp})$ where $P_{n-2}^\mathrm{opp}$ denotes the opposite graph of the path $\mathrm{P}_{n-2}$ of length $n-2$. 

\medskip \noindent
Flat braid groups have been introduced in the literature under various names, often independently. For instance, one can meet \emph{Grothendieck cartographical groups} in \cite{MR1106916, Flat}; \emph{traid groups} in \cite{MR4035955} for applications to physics; \emph{twin groups} in \cite{MR1386845, MR1370644} in connection with \emph{doodles}; \emph{pseudo-braids groups} in \cite{MR4071367} when studying \emph{diagram groups} (see also \cite{Farley}); \emph{planar braid groups} in \cite{MR4079623, mostovoy}. For notational convenience, we call our groups \emph{flat braid groups} in reference to flat braids from \cite{MR1738392}. 

\medskip \noindent
Several articles available in the literature investigate the algebraic structure of (pure) flat braid groups. See for instance \cite{MR3943376} about presentations and ranks of the commutator subgroups of flat braid groups; \cite{MR4027588} about similar results for pure flat braid groups; \cite{MR4079623} about the algebraic structures of pure flat braid groups on $\leq 6$ strands; \cite{MR4145210} about the conjugacy problem and the structure of automorphism groups of flat braid groups; \cite{MR4192499} about Property $R_\infty$ of flat braid groups; and \cite{MR4701892} about a connection with the so-called \emph{cactus groups} (see also \cite{Cactus}).

\medskip \noindent
Motivated by the fact that pure virtual twin groups are right-angled Artin groups \cite{MR4651964} and by the structure of pure flat braid groups on $\leq 6$ strands \cite{MR4079623}, it has been conjectured that:

\begin{conj}[\cite{MR4739232}]\label{conj}
For every $n \geq 3$, $\mathrm{PFB}_n$ is a right-angled Artin group. 
\end{conj}

\noindent
In this article, we disprove this conjecture by showing that most pure flat braid groups are not right-angled Artin groups, even up to a finite index.

\begin{thm}\label{thm:MainThm}
For every $n=7$ or $\geq 11$, the pure flat braid group $\mathrm{PFB}_n$ is not virtually a right-angled Artin group.
\end{thm}

\noindent
Despite the fact that Conjecture~\ref{conj} is false, we would like to emphasize that the following (vague) question is still interesting: to which extent does a (pure) flat braid group look like a right-angled Artin group?  For instance, Theorem~\ref{thm:MainThm} shows that $\mathrm{FB}_n$ does not contain a right-angled Artin group of finite index in $\mathrm{PFB}_n$, but is there such a subgroup elsewhere in $\mathrm{FB}_n$? More precisely:

\begin{question}\label{QuestionOne}
Is $\mathrm{FB}_n$ virtually a right-angled Artin group? 
\end{question} 

\noindent
For instance, a natural candidate to check would be the commutator subgroup of $\mathrm{FB}_n$, which is torsion-free. In case of a negative answer to Question~\ref{QuestionOne}, it would be natural to ask whether $\mathrm{FB}_n$ contains a finite-index subgroup that, despite not being a right-angled Artin group, turns out to be isomorphic to a finite-index subgroup of some right-angled Artin groups. In other words:

\begin{question}\label{QuestionTwo}
Is $\mathrm{FB}_n$ commensurable (or at least quasi-isometric) to a right-angled Artin group?
\end{question}

\noindent
Since, as already said, flat braid groups are right-angled Coxeter groups, Question~\ref{QuestionTwo} is a particular case of Question~\ref{question:MainQuestion}. Perhaps surprisingly, in view on Theorem~\ref{thm:MainThm}, we answer positively Question~\ref{QuestionTwo} for $n=7$. 

\begin{thm}\label{thm:Comm}
The flat braid group $\mathrm{FB}_7$ is commensurable to the right-angled Artin group $A(P_4)$. 
\end{thm}

\noindent
In order to prove Theorem~\ref{thm:Comm}, we exploit the observation that $\mathrm{FB}_7$ and $A(P_4)$ are both (virtually) fundamental groups of compact flip manifolds. It is known by \cite{MR2376814} that right-angled Artin groups defined by finite trees can be described as fundamental groups of such $3$-manifolds. For $\mathrm{FB}_7$, we start by proving that $\mathrm{FB}_7$ contains an index-$8$ normal subgroup isomorphic to an index-$2$ subgroup of $\langle a,t \mid [a,tat^{-1}]=1 \rangle$ (see Section~\ref{section:Lampraag}). This group can be easily described as the fundamental group of a compact flip manifold $M_2$. Given another such $3$-manifold $M_1$ whose fundamental group is isomorphic to $A(P_4)$, we construct a common finite-sheeted cover $M_0 \to M_1,M_2$ (see Section~\ref{section:Flip}), which allows us to deduce Theorem~\ref{thm:Comm}.

\paragraph{About the proof of Theorem~\ref{thm:MainThm}.} The first step is to prove the theorem for the flat braid group on seven strands, namely:

\begin{thm}\label{thm:FirstPart}
The flat braid group $\mathrm{PFB}_7$ is not virtually a right-angled Artin group.
\end{thm}

\noindent
This is the smallest number of strands for which the statement holds (see \cite{MR4145210}). Theorem~\ref{thm:FirstPart} is proved as follows. First, we observe that, since $\mathrm{FB}_7$ does not contain a subgroup isomorphic to $\mathbb{Z}^3$, to $\mathbb{F}_2 \times \mathbb{F}_2$, or to the fundamental group of a closed surface of genus $\geq 2$ (Lemma~\ref{lem:NoSubgroup}), every right-angled Artin group that appears as a subgroup of $\mathrm{FB}_7$ is defined by a forest (Corollary~\ref{cor:SubRAAG}). As a consequence, it suffices to show that $\mathrm{PFB}_7$ does not contain as a finite-index subgroup a right-angled Artin group $A(T)$ defined by a finite tree $T$. For this, we define the \emph{thick subgroup} $\mathrm{Thick}(G)$ of a group $G$ as the subgroup generated by the centralisers of all the elements whose centralisers are not virtually abelian. Then, by showing that $\mathrm{Thick}(A(T)) = A(T)$ (Lemma~\ref{lem:ModThickRAAG}) but that $\mathrm{Thick}(\mathrm{PFB}_7)$ has infinite index in $\mathrm{PFB}_7$ (Lemma~\ref{lem:ModThickFB}), we conclude that $A(T)$ cannot be a finite-index subgroup of $\mathrm{PFB}_7$. 

\medskip \noindent
Once Theorem~\ref{thm:FirstPart} is proved, it is not so difficult to deduce that $\mathrm{PFB}_n$ is not a right-angled Artin group for $n \geq 11$. The trick is that $\mathrm{PFB}_n$ contains a natural copy of $\mathrm{PFB}_7 \times \mathrm{PFB}_{n-7}$, which turns out to be a maximal product subgroup. But, in a right-angled Artin group, maximal product subgroups are well-understood; their factors, in particular, are also right-angled Artin groups. Since this product decomposition is unique (Corollary~\ref{cor:UniqueFactor}), it follows that, if $\mathrm{PFB}_n$ were a right-angled Artin group, then $\mathrm{PBF}_7$ would be a right-angled Artin group as well, contradicting Theorem~\ref{thm:FirstPart}. 

\medskip \noindent
In order to prove that $\mathrm{PFB}_n$ is not virtually a right-angled Artin group, the strategy is basically the same, but some technicality is required. In Section~\ref{section:IMC}, we introduce the notion of \emph{IMC generating sets}, which is of independent interest, and we prove that, under the assumption that our groups admit IMC generating sets, factors of finite-index subgroups in products must be contained in factors of the whole product (Lemma~\ref{lem:IMCandProduct}).

\paragraph{Acknowledgements.} I am grateful to Hoel Queffelec for his comments on a preliminary version of this paper, and to P. Bellingeri and N. Nanda for interesting discussions regarding flat braid groups.

\section{Preliminaries on graph products}\label{section:GP}

\noindent
Let $\Gamma$ be a graph and $\mathcal{G}= \{ G_u \mid u \in \Gamma \}$ a collection of groups indexed by the vertices of $\Gamma$. The \emph{graph product} $\Gamma \mathcal{G}$ is 
$$\langle G_u \ (u\in \Gamma) \mid [G_u,G_v]=1 \ (\{u,v\} \in E(\Gamma)) \rangle$$
where $E(\Gamma)$ denotes the edge-set of $\Gamma$ and where $[G_u,G_v]=1$ is a shorthand for $[g,h]=1$ for all $g \in G_u$, $h \in G_v$. The groups of $\mathcal{G}$ are referred to as the \emph{vertex-groups}. 

\medskip \noindent
Graph products of groups will allow us to state and prove results simultaneously about right-angled Artin groups (i.e.\ when vertex-groups are infinite cyclic) and about right-angled Coxeter groups (i.e.\ when vertex-groups are cyclic of order two).

\medskip \noindent
\textbf{Convention.} We always assume that the groups in $\mathcal{G}$ are non-trivial. Notice that it is not a restrictive assumption, since a graph product with some trivial factors can be described as a graph product over a smaller graph all of whose factors are non-trivial.

\medskip \noindent
A \emph{word} in $\Gamma \mathcal{G}$ is a product $g_1 \cdots g_n$ where $n \geq 0$ and where, for every $1 \leq i \leq n$, $g_i \in G$ for some $G \in \mathcal{G}$; the $g_i$'s are the \emph{syllables} of the word, and $n$ is the \emph{length} of the word. Clearly, the following operations on a word does not modify the element of $\Gamma \mathcal{G}$ it represents:
\begin{description}
	\item[Cancellation:] delete the syllable $g_i$ if $g_i=1$;
	\item[Amalgamation:] if $g_i,g_{i+1} \in G$ for some $G \in \mathcal{G}$, replace the two syllables $g_i$ and $g_{i+1}$ by the single syllable $g_ig_{i+1} \in G$;
	\item[Shuffling:] if $g_i$ and $g_{i+1}$ belong to two adjacent vertex-groups, switch them.
\end{description}
A word is \emph{graphically reduced} if its length cannot be shortened by applying these elementary moves. Every element of $\Gamma \mathcal{G}$ can be represented by a graphically reduced word, and this word is unique up to the shuffling operation. This allows us to define the \emph{support} of an element as the subgraph of $\Gamma$ induced by the vertices labelling the syllables of a graphically reduced word representing our element. Similarly, a word is \emph{graphically cyclically reduced} if all its cyclic permutations are graphically reduced. An element of $\Gamma \mathcal{G}$ that can be represented by a graphically cyclically reduced word is \emph{graphically cyclically reduced}. Every element is conjugate to a graphically cyclically reduced word. One can define the \emph{essential support} of an element as the support of a graphically cyclically reduced element that is conjugate to it. For more information on graphically reduced words, we refer to \cite{GreenGP} (see also \cite{HsuWise,VanKampenGP}).

\paragraph{Parabolic subgroups.} Let $\Gamma$ be a graph and $\mathcal{G}$ a collection of groups indexed by $\Gamma$. Given a subgraph $\Lambda \subset \Gamma$, we denote by $\langle \Lambda \rangle$ the subgroup of $\Gamma \mathcal{G}$ generated by the vertex-groups labelling the vertices of $\Lambda$. A subgroup of $\Gamma \mathcal{G}$ of the form $g \langle \Phi \rangle g^{-1}$ for some element $g \in \Gamma \mathcal{G}$ and subgraph $\Phi \subset \Gamma$ is a \emph{parabolic subgroup}. Let us record a few basic results about parabolic subgroups that will be useful later. 

\begin{lemma}\label{lem:WhenVirtuallyZ}
Let $\Gamma$ be a graph and $\mathcal{G}$ a collection of groups indexed by $\Gamma$. The graph product $\Gamma \mathcal{G}$ is virtually cyclic if and only if one of the following conditions hold:
\begin{itemize}
	\item $\Gamma$ is a complete graph all of whose vertices are labelled by finite groups;
	\item $\Gamma = \Xi \ast \{u\}$ where $\Xi$ is a complete graph all of whose vertices are labelled by finite groups and where $u$ is a vertex labelled by a virtually-$\mathbb{Z}$ group;
	\item $\Gamma = \Xi \ast \{u,v\}$ where $\Xi$ is a complete graph all of whose vertices are labelled by finite groups and where $u,v$ are two non-adjacent vertices both labelled by $\mathbb{Z}_2$. 
\end{itemize}
\end{lemma}

\noindent
Recall that, given two graphs $\Phi$ and $\Psi$, their \emph{join} $\Phi \ast \Psi$ is the graph obtained from $\Phi \sqcup \Psi$ by connecting with an edge every vertex of $\Phi$ with every vertex of $\Psi$. 

\begin{proof}[Proof of Lemma~\ref{lem:WhenVirtuallyZ}.]
If $\Gamma$ is complete, then $\Gamma \mathcal{G}$ is the product of its vertex-groups. In this case, $\Gamma \mathcal{G}$ is virtually cyclic if and only if either all its vertex-groups are finite or one vertex-group is virtually-$\mathbb{Z}$ and all its other vertex-groups are finite. This corresponds to the first and second items of our lemma.

\medskip \noindent
Now, assume that $\Gamma$ is not complete. Fix two non-adjacent vertices $u,v \in \Gamma$. Since the subgroup $\langle u,v \rangle$ of $\Gamma \mathcal{G}$ decomposes as the free product $\langle u \rangle \ast \langle v \rangle$, the only possibility for $\Gamma \mathcal{G}$ to be virtually cyclic is that $\langle u \rangle$ and $\langle v \rangle$ are both cyclic of order two. Next, if $\Gamma$ contains a vertex $w$ that is not adjacent to both $u$ and $v$, then $\langle u,v,w \rangle$ decomposes as one of the following free products: $\langle u \rangle \ast \langle v \rangle \ast \langle w \rangle$, $(\langle u \rangle \times \langle w \rangle) \ast \langle v \rangle$, or $\langle u \rangle \ast (\langle v \rangle \times \langle w \rangle)$. In any case, $\Gamma \mathcal{G}$ cannot be virtually cyclic. Therefore, $\Gamma$ decomposes as a join $\Xi \ast \{ u,v\}$. Algebraically, this implies that $\Gamma \mathcal{G}$ decomposes as the product $\langle \Xi \rangle \times \langle u,v\rangle \simeq \langle \Xi \rangle \ast \mathbb{D}_\infty$. Then, $\Gamma \mathcal{G}$ is virtually cyclic if and only if $\langle \Xi \rangle$ is finite, which amounts to saying that $\Xi$ is a complete graph all of whose vertices are labelled by finite groups. This corresponds to the third item of our lemma. 
\end{proof}

\begin{lemma}\label{lem:WhenFiniteIndex}
Let $\Gamma$ be a graph and $\mathcal{G}$ a collection of groups indexed by $\Gamma$. For every subgraph $\Lambda \subset \Gamma$, $\langle \Lambda \rangle$ has finite index in $\Gamma \mathcal{G}$ if and only if $\Gamma= \Lambda \ast \Xi$ where $\Xi$ is a complete graph all of whose vertices are labelled by finite groups. 
\end{lemma}

\begin{proof}
First, assume that there exist two non-adjacent vertices $u \in \Gamma \backslash \Lambda$ and $v \in \Lambda$. Fix two non trivial elements $a \in \langle u \rangle$ and $b \in \langle v \rangle$. If there exist two distinct integers $p,q \geq 1$ such that $(ba)^p$ and $(ba)^q$ belong to the same $\langle \Lambda \rangle$-coset, then we find an integer $\ell \geq 1$ such that $(ba)^\ell \in \langle \Lambda \rangle$. But this is impossible since, as a graphically reduced word, $(ba)^\ell$ cannot represent an element of $\langle \Lambda \rangle$. Consequently, $\langle \Lambda \rangle$ must have infinite index in $\Gamma \mathcal{G}$.

\medskip \noindent
From now on, assume that every vertex in $\Gamma \backslash \Lambda$ is adjacent to every vertex in $\Lambda$, i.e.\ $\Gamma$ decomposes as a join $\Lambda \ast \Xi$. Algebraically, $\Gamma \mathcal{G}$ decomposes as the product $\langle \Lambda \rangle \ast \langle \Xi \rangle$. Then, $\langle \Xi \rangle$ has finite index in $\Gamma \mathcal{G}$ if and only if $\langle \Xi \rangle$ is finite, which amounts to saying that $\Xi$ is a complete graph all of whose vertices are labelled by finite groups.
\end{proof}

\begin{lemma}\label{lem:ParabolicInclusion}
Let $\Gamma$ be a graph and $\mathcal{G}$ a collection of groups indexed by $\Gamma$. For all subgraphs $\Phi,\Psi \subset \Gamma$ and element $a \in \Gamma \mathcal{G}$, the inclusion $a \langle \Phi \rangle a^{-1} \leq \langle \Psi \rangle$ holds if and only if $\Phi \subset \Psi$ and $a \in \langle \Psi \rangle \cdot \langle \mathrm{star}(\Phi) \rangle$.
\end{lemma}

\noindent
A proof of this lemma can be found in \cite[Lemma~3.17]{MR4295519}. As an immediate consequence, it follows that:

\begin{cor}\label{cor:InclusionParabolic}
Let $\Gamma$ be a graph and $\mathcal{G}$ a collection of groups indexed by $\Gamma$. For all subgraph $\Phi \subset \Gamma$ and element $a \in \Gamma \mathcal{G}$, if $a \langle \Phi \rangle a^{-1} \leq \langle \Phi \rangle$ then $a \langle \Phi \rangle a^{-1} = \langle \Phi \rangle$. 
\end{cor}

\paragraph{Join-subgroups.} Parabolic subgroups given by join subgraphs play a central role in the study of graph products of groups. In particular, as explained by our next lemma, they can be used in order to characterise maximal product subgroups in graph products. Here, we refer to a \emph{maximal product subgroup} of a given group as a maximal member of the collection of the subgroups that decompose as products of two non-trivial groups. 

\begin{lemma}\label{lem:MaxProdInGP}
Let $\Gamma$ be a finite graph and $\mathcal{G}$ a collection of groups indexed by $\Gamma$. A subgroup of $\Gamma \mathcal{G}$ is a maximal product subgroup if and only if either it is conjugate to $\langle \Lambda \rangle$ for some maximal join $\Lambda \subset \Gamma$ or it is a maximal product subgroup in a conjugate of a vertex-group given by an isolated vertex of $\Gamma$. 
\end{lemma}

\noindent
A proof of this lemma can be found in \cite[Proposition~2.8]{MR4808711}. It is worth noticing that the maximality given by the previous lemma behaves nicely with respect to commensurability:

\begin{lemma}\label{lem:GPcontainedInProduct}
Let $\Gamma$ be a finite graph and $\mathcal{G}$ a collection of groups indexed by $\Gamma$. A subgroup of $\Gamma \mathcal{G}$ commensurable to a product of two infinite groups is contained in a conjugate of a vertex-group or in a maximal product subgroup.
\end{lemma}

\begin{proof}
Let $H \leq \Gamma \mathcal{G}$ be a subgroup that contains a finite-index subgroup $\dot{H}$ in common with a product $A \times B$ of two infinite groups. Given an $h \in H$, there exists some $p \geq 1$ such that $h^p \in \dot{H}$. Thinking of $h^p$ as an element of $A \times B$, it can be written as $(a,b)$ for some $a \in A$ and $b \in B$. Fix a $q \geq 1$ such that $a^q$ and $b^q$ either are trivial or have infinite order. Then the centraliser of $h^{pq}= (a^q,b^q)$ in $A \times B$ is $A \times B$ if $a^q=b^q=1$; contains $A \times \langle b \rangle \simeq A \times \mathbb{Z}$ if $a^q=1$ and if $b$ has infinite order; contains $\langle a \rangle \times B \simeq \mathbb{Z} \times B$ if $a$ has infinite order and if $b^q=1$; contains $\langle a \rangle \times \langle b \rangle \simeq \mathbb{Z}^2$ if $a$ and $b$ both have infinite order. In any case, the centraliser of $h^{pq}$ in $A \times B$, and a fortiori in $H$ is not virtually cyclic. Thus, we have proved that every element of $H$ has a non-trivial power whose centraliser is not virtually cyclic.

\medskip \noindent
Now, assume that $H$ is not contained in a conjugate of a vertex-group nor in a maximal product subgroup. We claim that $H$ must contain an element all of whose non-trivial powers have a virtually cyclic centraliser. As a consequence of the previous observation, this will conclude the proof of our lemma. Let $g \langle \Lambda \rangle g^{-1}$ denote the smallest parabolic subgroup containing $H$. (Such a subgroup exists according to \cite{MR3365774}.) According to \cite[Corollary~6.20]{MR3368093}, there exists a tree on which $g \langle \Lambda \rangle g^{-1}$ acts such that $H$ contains a WPD element $h \in H$. Every non-trivial power of such an element must have a virtually cyclic centralisers in $g \langle \Lambda \rangle g^{-1}$. But, as a consequence of Proposition~\ref{prop:CentralisersGP} below, the centraliser of an element of $H$ in $\Gamma \mathcal{G}$ decomposes as the product of the centraliser of the element in $g \langle \Lambda \rangle g^{-1}$ with $g \langle \mathrm{link}(\Lambda) \rangle g^{-1}$. Since $H$ cannot be contained in a join-subgroup, $\mathrm{link}(\Lambda) \rangle$ must be empty. Consequently, the centraliser of a non-trivial power of $h$ in $\Gamma \mathcal{G}$, and a fortiori in $H$, is virtually cyclic. 
\end{proof}

\noindent
Finally, let us observe no two maximal product subgroups in graph products can be commensurable:

\begin{lemma}\label{lem:BigInterProd}
Let $\Gamma$ be a graph and $\mathcal{G}$ a collection of groups indexed by $\Gamma$. Let $\Phi,\Psi \subset \Gamma$ be two maximal joins and let $g,h \in \Gamma \mathcal{G}$. If $g \langle \Phi \rangle g^{-1}$ has a finite-index subgroup contained in $h \langle \Psi \rangle h^{-1}$, then $g \langle \Phi \rangle g^{-1}= h \langle \Psi \rangle h^{-1}$. 
\end{lemma}

\begin{proof}
According to \cite{MR3365774} (see also \cite{Mediangle}), an intersection of two parabolic subgroups is again a parabolic subgroup, so there exist $k \in \Gamma \mathcal{G}$ and $\Lambda \subset \Gamma$ such that
$$g \langle \Phi \rangle g^{-1 } \cap h \langle \Psi \rangle h^{-1}= k \langle \Lambda \rangle k^{-1}.$$
Since $g \langle \Phi \rangle g^{-1}$ contains a finite-index subgroup that is also contained in $h \langle \Psi \rangle h^{-1}$, necessarily $k \langle \Lambda \rangle k^{-1}$ has finite index in $g \langle \Phi \rangle g^{-1}$. Notice that, as a consequence of Lemma~\ref{lem:ParabolicInclusion}, $\Lambda \subset \Phi$ and $g^{-1}k$ belongs to $\langle \mathrm{star}(\Phi) \rangle$. This implies that $\langle \Lambda \rangle$ is conjugate in $\langle \Phi \rangle$ (by an element of $\langle \Phi \rangle$) to a finite-index subgroup. Since $\langle \Lambda \rangle$ then must have finite index in $\langle \Phi \rangle$, we deduce from Lemma~\ref{lem:WhenFiniteIndex} that $\Phi$ decomposes as a join with $\Lambda$ as a factor. But $\Phi$ is by assumption a maximal join in $\Gamma$, so we must have $\Phi = \Lambda$. So far, we have proved that
$$g \langle \Phi \rangle g^{-1 } \cap h \langle \Psi \rangle h^{-1}= k \langle \Phi \rangle k^{-1}.$$
From the inclusion $k \langle \Phi \rangle k^{-1} \leq g \langle \Phi \rangle g^{-1}$, we deduce from Corollary~\ref{cor:InclusionParabolic} that $k \langle \Phi \rangle k^{-1}= g \langle \Phi \rangle g^{-1}$. Then, the centred equality above amounts to saying that $g \langle \Phi \rangle g^{-1}  \leq h \langle \Psi \rangle h^{-1}$. We know from Lemma~\ref{lem:ParabolicInclusion} that $\Phi \subset \Psi$. But $\Psi$ is a maximal join in $\Gamma$, so we must have $\Phi = \Psi$. Then, we conclude from Corollary~\ref{cor:InclusionParabolic} that $g \langle \Phi \rangle g^{-1} =  h \langle \Psi \rangle h^{-1}$, as desired.
\end{proof}

\paragraph{Centralisers in graph products.} For future use, we record the structure of centralisers in graph products of groups:

\begin{prop}[\cite{MR2320444}]\label{prop:CentralisersGP}
Let $\Gamma$ be a graph and $\mathcal{G}$ a collection of groups indexed by $\Gamma$. Fix a graphically cyclically reduced element $g \in \Gamma \mathcal{G}$ and decompose its support as a join 
$$\mathrm{supp}(g)= \Phi_1 \ast \cdots \ast \Phi_r \ast \Psi_1 \ast \cdots \ast \Psi_s$$ 
such that $\Phi_i$ is reduced to a single vertex for every $1 \leq i \leq r$ and such that $\Psi_i$ is an irreducible subgraph with at least two vertices for every $1 \leq i \leq s$. Write $g$ as a graphically reduced word $a_1 \cdots a_r \cdot b_1 \cdots b_s$ such that $\mathrm{supp}(a_i)= \Phi_i$ for every $1 \leq i \leq r$ and $\mathrm{supp}(b_i)= \Psi_i$ for every $1 \leq i \leq s$. Then the centraliser of $g$ in $\Gamma \mathcal{G}$ is
$$C_{\Gamma \mathcal{G}} (g)= C_{\langle \Phi_1 \rangle}(a_1) \times \cdots \times C_{\langle \Phi_r \rangle}(a_r) \times \langle h_1 \rangle \times \cdots \times \langle h_s \rangle \times \langle \mathrm{link}(\mathrm{supp}(g)) \rangle$$
where each $h_i$ is a primitive element of $\langle \Psi_i \rangle$ such that $b_i \in \langle h_i \rangle$. 
\end{prop}

\noindent
As a first application of this description of centralisers, we observe that, in many graph products, the centraliser of a non-trivial power of an element is pretty much the same as the centraliser of the element itself. More formally:

\begin{definition}
A group $G$ has \emph{stable centralisers} if, for all $g \in G$ and $k \geq 1$, $C(g)$ equals $C(g^k)$. The group has \emph{almost stable centralisers} if, for all $g \in G$ and $k \geq 1$, $C(g)$ has finite index in $C(g^k)$. 
\end{definition}

\noindent
As an easy application of Proposition~\ref{prop:CentralisersGP}:

\begin{lemma}\label{lem:GPstableCent}
Let $\Gamma$ be a graph and $\mathcal{G}$ a collection of torsion-free groups indexed by $\Gamma$. If the groups in $\mathcal{G}$ have (almost) stable centralisers, then $\Gamma \mathcal{G}$ has (almost) stable centralisers. 
\end{lemma}

\begin{proof}
Let $g \in \Gamma \mathcal{G}$ be an element. Up to conjugating $g$, we assume for convenience that it is graphically cyclically reduced. Decompose $\mathrm{supp}(g)$ as a join $\Phi_1 \ast \cdots \ast \Phi_r \ast \Psi_1 \ast \cdots \ast \Psi_s$ and write $g$ as a product $a_1 \cdots a_r \cdot b_1 \cdots b_s$ as in Proposition~\ref{prop:CentralisersGP}. So 
$$C_{\Gamma \mathcal{G}} (g)= C_{\langle \Phi_1 \rangle}(a_1) \times \cdots \times C_{\langle \Phi_r \rangle}(a_r) \times \langle h_1 \rangle \times \cdots \times \langle h_s \rangle \times \langle \mathrm{link}(\mathrm{supp}(g)) \rangle$$
Given a $k \geq 1$, because vertex-groups are torsion-free, Proposition~\ref{prop:CentralisersGP} applies to the decomposition $g^k=a_1^k \cdots a_r^k\cdot b_1^k \cdots b_s^k$ and shows that 
$$C_{\Gamma \mathcal{G}} (g^k)= C_{\langle \Phi_1 \rangle}(a_1^k) \times \cdots \times C_{\langle \Phi_r \rangle}(a_r^k) \times \langle h_1 \rangle \times \cdots \times \langle h_s \rangle \times \langle \mathrm{link}(\mathrm{supp}(g)) \rangle.$$
If vertex-groups have stable centralisers (resp.\ almost stable centralisers), then each $C_{\langle \Phi_i \rangle} (a_i)$ agrees with (resp.\ has finite index in) $C_{\langle \Phi_i \rangle}(a_i^k)$. This implies that the centraliser of $g^k$ in $\Gamma \mathcal{G}$ agrees with (resp.\ has finite index in) the centraliser of $g$ in $\Gamma \mathcal{G}$. 
\end{proof}

\begin{cor}\label{cor:StableCent}
Right-angled Artin groups, as well as their subgroups, have stable centralisers.
\end{cor}

\noindent
It is worth mentioning that there exist also many graph products that do not have almost stable centralisers. This is the case, for instance, of the product $\mathbb{D}_\infty \times \mathbb{D}_\infty$ of two infinite dihedral groups. Indeed, given an infinite-order element $a \in \mathbb{D}_\infty$ and an element of order two $b \in \mathbb{D}_\infty$, the centraliser of $(a,b)$ is isomorphic to $\mathbb{Z} \times \mathbb{Z}_2$ but the centraliser of its square $(a^2,1)$ is isomorphic to $\mathbb{Z} \times \mathbb{Z}$.

\paragraph{Virtual centres.} Motivating by the observation that the centres of a group and of its finite-index subgroups may be quite different, we introduce the following definition:

\begin{definition}
The \emph{virtual centre} $\mathrm{VZ}(G)$ of a group $G$ is the set of the elements that centralise some finite-index subgroups of $G$.
\end{definition}

\noindent
Notice that:

\begin{lemma}
The virtual centre of a group is a normal subgroup.
\end{lemma}

\begin{proof}
Let $G$ be a group and $a,b \in G$ two elements. Notice that the following two observations hold:
\begin{itemize}
	\item If $a$ (resp.\ $b$) commutes with all the elements of a finite-index subgroup $H \leq G$ (resp.\ $K \leq G$), then $ab^{-1}$ commutes with all the elements of the finite-index subgroup $H \cap K$. Hence $ab^{-1} \in \mathrm{VZ}(G)$. 
	\item If $a$ commutes with all the elements of a finite-index subgroup $H \leq G$, then $bab^{-1}$ commutes with all the elements of the finite-index subgroup $bHb^{-1}$. Hence $bab^{-1} \in \mathrm{VZ}(G)$. 
\end{itemize}
We conclude that $\mathrm{VZ}(G)$ is indeed a normal subgroup of $G$. 
\end{proof}

\noindent
For future use, let us describe virtual centres of graph products of groups:

\begin{lemma}\label{lem:VirtualCentreGP}
Let $\Gamma$ be a finite graph and $\mathcal{G}$ a collection of groups indexed by $\Gamma$. Decompose $\Gamma$ as a join
$$\Gamma= \{ u_1 \} \ast \cdots \ast \{u_r\} \ast \{a_1,b_1\} \ast \cdots \ast \{a_s,b_s\} \ast \Lambda_1 \ast \cdots \ast \Lambda_n$$
where $u_1, \ldots, u_r \in \Gamma$ are single vertices, where each $a_i,b_i \in \Gamma$ are two non-adjacent vertices labelled by $\mathbb{Z}_2$, and where each $\Lambda_i$ is an irreducible subgraph containing at least two vertices not both labelled by $\mathbb{Z}_2$. Then
$$\mathrm{VZ}(\Gamma \mathcal{G})= \mathrm{VZ}(\langle u_1 \rangle) \times \cdots \times \mathrm{VZ}(\langle u_r \rangle) \times \langle a_1b_1 \rangle \times \cdots \times \langle a_sb_s \rangle.$$
\end{lemma}

\begin{proof}
Let $g \in \Gamma \mathcal{G}$ be an element that belongs to the virtual centre of $\Gamma \mathcal{G}$. This amounts to saying that the centraliser of $g$ has finite index in $\Gamma \mathcal{G}$. Our goal is to prove that $g$ belongs to the subgroup
$$V:= \mathrm{VZ}(\langle u_1 \rangle) \times \cdots \times \mathrm{VZ}(\langle u_r \rangle) \times \langle a_1b_1 \rangle \times \cdots \times \langle a_sb_s \rangle.$$
Because $V$ clearly lies in the virtual centre of $\Gamma \mathcal{G}$, this will be sufficient to conclude the proof of our lemma. Notice also that $V$ is a normal subgroup, so, up to conjugating $g$, we assume without loss of generality that $g$ is graphically cyclically reduced. Following Proposition~\ref{prop:CentralisersGP}, decompose the support of $g$ as a join 
$$\mathrm{supp}(g)= \Phi_1 \ast \cdots \ast \Phi_p \ast \Psi_1 \ast \cdots \ast \Psi_q$$ 
such that $\Phi_i$ is reduced to a single vertex for every $1 \leq i \leq p$ and such that $\Psi_i$ is an irreducible subgraph with at least two vertices for every $1 \leq i \leq q$. Write $g$ as a graphically reduced word $a_1 \cdots a_p \cdot b_1 \cdots b_q$ such that $\mathrm{supp}(a_i)= \Phi_i$ for every $1 \leq i \leq p$ and $\mathrm{supp}(b_i)= \Psi_i$ for every $1 \leq i \leq q$. According to Proposition~\ref{prop:CentralisersGP}, the centraliser of $g$ in $\Gamma \mathcal{G}$ is
$$C_{\Gamma \mathcal{G}} (g)= C_{\langle \Phi_1 \rangle}(a_1) \times \cdots \times C_{\langle \Phi_r \rangle}(a_p) \times \langle h_1 \rangle \times \cdots \times \langle h_q \rangle \times \langle \mathrm{link}(\mathrm{supp}(g)) \rangle$$
where each $h_i$ is a primitive element of $\langle \Psi_i \rangle$ such that $b_i \in \langle h_i \rangle$. Since this centraliser has finite index in $\Gamma \mathcal{G}$, the following assertions hold:
\begin{itemize}
	\item for every $1 \leq i \leq p$, $C_{\langle \Phi_i \rangle}(a_i)$ has finite-index in $\langle \Phi_i \rangle$, i.e.\ $a_i \in \mathrm{VZ}(\langle \Phi_i \rangle)$;
	\item $\langle h_i \rangle$ has finite index in $\langle \Psi_i \rangle$ for every $1 \leq i \leq q$;
	\item $\langle \Phi_1 \cup \cdots \cup \Phi_p \cup \Psi_1 \cup \cdots \cup \Psi_q \cup \mathrm{link}(\mathrm{supp}(g)) \rangle$ has finite index in $\Gamma \mathcal{G}$.
\end{itemize}
The second item implies that each $\langle \Psi_i \rangle$ is virtually cyclic. According to Lemma~\ref{lem:WhenVirtuallyZ}, the only possibility is that $\Psi_i$ is given by two non-adjacent vertices both labelled by $\mathbb{Z}_2$. And the third item implies, according to Lemma~\ref{lem:WhenFiniteIndex}, that 
$$\Gamma = \Phi_1 \ast \cdots \ast \Phi_r \ast \Psi_1 \ast \cdots \ast \Psi_s \ast \mathrm{link}(\mathrm{supp}(g)).$$
It follows that $p=r$ and $q=s$; that $\mathrm{link}(\mathrm{supp}(g)) = \Lambda_1 \ast \cdots \ast \Lambda_n$; and that, up to reordering our subgraphs, $\Phi_i = \{u_i\}$ for every $1 \leq i \leq r$ and $\Psi_i = \{ a_i,b_i\}$ for every $1 \leq i \leq s$. Notice that, for every $1 \leq i \leq s$, $b_i$ is an infinite-order element of the infinite dihedral group $\langle \Psi_i \rangle = \langle a_i,b_i \rangle$, so we can take $h_i = a_ib_i$. We conclude that $g$ indeed belongs to $V$, as desired. 
\end{proof}

\begin{cor}\label{cor:FBVZ}
For every $k \geq 4$, the $\mathrm{FB}_k$ and $\mathrm{PFB}_k$ have trivial virtual centres.
\end{cor}

\begin{proof}
Thinking of the flat braid group $\mathrm{FB}_k$ as the right-angled Coxeter group $C(P^{\mathrm{opp}}_{k-2})$, it follows immediately from Lemma~\ref{lem:VirtualCentreGP} that $\mathrm{FB}_k$ has a trivial virtual centre as soon as $k \geq 4$. Then, notice that an element of the virtual centre of $\mathrm{PFB}_k$ also belongs to the virtual centre of $\mathrm{FB}_k$, so it must be trivial. 
\end{proof}

\paragraph{Acylindrical hyperbolicity.} In order to prove Corollaries~\ref{cor:PFBnotProd} and~\ref{cor:GPimc}, we will use a few tools coming from the theory of \emph{acylindrically hyperbolic groups}. We refer the reader to \cite{MR3966794} and references therein for more information on the subject. Recall an acylindrically hyperbolic group $G$ has a unique maximal finite normal subgroup, referred to as its \emph{finite radical}  \cite[Theorem~2.24]{MR3589159}.

\begin{prop}\label{prop:GPacyl}
Let $\Gamma$ be a finite graph and $\mathcal{G}$ a collection of groups indexed by $\Gamma$. If $\Gamma$ is not a join and has at least two vertices, then $\Gamma \mathcal{G}$ is acylindrically hyperbolic and its finite radical is trivial.
\end{prop}

\begin{proof}
The acylindrical hyperbolicity of $\Gamma \mathcal{G}$ is given by \cite[Corollary~2.13]{MR3368093}. It remains to verify that the finite radical $R$ of $\Gamma \mathcal{G}$ is trivial. Let $g \in \Gamma \mathcal{G}$ be an element and $\Theta \subset \Gamma$ a subgraph such that $g \langle \Theta \rangle g^{-1}$ is the unique smallest parabolic subgroup containing $R$ (see \cite[Proposition~3.10]{MR3365774}). As it is well-known that finite subgroups in graph products are contained in clique-subgroups (see for instance \cite[Theorem~2.115 and Corollary~8.7]{QM} for a geometric proof), we deduce from Lemma~\ref{lem:ParabolicInclusion} that $\Theta$ is complete. Moreover, since $\Gamma \mathcal{G}$ normalises $R$, necessarily $g\langle \Theta \rangle g^{-1}$ is normalised by $\Gamma \mathcal{G}$. But, according to \cite[Proposition~3.13]{MR3365774}, the normaliser of $g \langle \Theta \rangle g^{-1}$ is $g \langle \mathrm{star}(\Theta) \rangle g^{-1}$. It follows from Corollary~\ref{cor:InclusionParabolic} that $\Gamma= \mathrm{star}(\Theta)$. Since $\Gamma$ is not a join and contains at least two vertices, this implies that $\Theta= \emptyset$. In other words, $R$ must be trivial, as desired. 
\end{proof}

\begin{cor}\label{cor:PFBnotProd}
For every $k \geq 4$, $\mathrm{PFB}_k$ is not virtually a product of two infinite groups. 
\end{cor}

\begin{proof}
For $k \geq 4$, it follows from Proposition~\ref{prop:GPacyl} that $\mathrm{FB}_k$, thought of as the right-angled Coxeter group $C(P_{k-2}^\mathrm{opp})$, is acylindrically hyperbolic. As a finite-index subgroup, $\mathrm{PFB}_k$ must be acylindrically hyperbolic as well. This prevents $\mathrm{PFB}_k$ from being virtually a product of two infinite groups, for instance as a consequence of \cite[Lemma~6.24]{MR3368093}.
\end{proof}

\section{Morphisms to products}\label{section:IMC}

\noindent
In this section, our goal is to show that a morphism between two products of groups satisfying mild assumptions has to send a factor to a factor. The main result in this direction is Lemma~\ref{lem:IMCandProduct}, using the notion of \emph{IMC generating sets} which we now define and study.

\subsection{IMC generating sets}

\noindent
In this section, we investigate a specific type of generating sets, namely:

\begin{definition}
Given an group $G$, $S \subset G$ is an IMC generating set if it satisfies the following conditions:
\begin{description}
	\item[(Independence)] for all distinct $s_1,s_2 \in S$ and all integers $p,q \geq 1$, $[s_1^p,s_2^q] \neq 1$;
	\item[(Maximal Centralisers)] for all $s \in S$ and $g \in G$, if $C(s) \subsetneq C(g)$ then $g=1$.
\end{description}
\end{definition}

\noindent
Our goal now is to show that most acylindrically hyperbolic groups admit IMC generating sets. Recall that every generalised loxodromic element $g \in G$ belongs to a unique maximal virtually cyclic subgroup of $G$, which we denote by $E(g)$ \cite[Lemma~6.5]{MR3589159}. According to \cite[Corollary~6.6]{MR3589159}, $E(g)$ coincides with $\{ h \in G \mid \exists n \geq 1, h g^n h^{-1}= g^{\pm n} \}$. As a consequence, $E(g^k)=E(g)$ for every $k \geq 1$. 

\begin{prop}\label{prop:AcylIMC}
Let $G$ be an acylindrically hyperbolic group. If the finite radical of $G$ is trivial, then $G$ admits an IMC generating set.
\end{prop}

\begin{proof}
Fix a non-elementary acylindrical action of $G$ on some hyperbolic space. Following \cite{MR3605030}, we refer to an element $g \in G$ as \emph{special} if $g$ is loxodromic and $E(g)=\langle g \rangle$. Let $S_0$ denote the set of all the special elements of $G$. Fix a set of representatives $S \subset S_0$ with respect to the equivalence relation: for all $r,s \in S_0$, $r \sim s$ whenever $r=s^{\pm 1}$. Let us verify that $S$ is an IMC generating set of $G$.

\medskip \noindent
According to \cite[Proposition~5.14]{MR3605030}, $S_0$ generates $G$, so $S$ is a generating set of $G$. 

\medskip \noindent
Now, let $s \in S$ and $g \in G$ be two elements satisfying $C(s) \subset C(g)$. Since $s$ and $g$ commute, we have
$$g \in C(s) \subset E(s) = \langle s \rangle.$$
Either $g=1$, in which case there is nothing to prove; or $g$ is a non-trivial power of $s$, which implies that 
$$C(s) \subset C(g) \subset E(g) = E(s) = \langle s \rangle \subset C(s),$$ 
hence $C(g)=C(s)$. 

\medskip \noindent
Next, let $r,s \in S$ and $p,q \geq 1$ be such that $[r^p,s^q]=1$. We have
$$r^p \in C(s^q) \subset E(s^q) = E(s) = \langle s \rangle,$$
from which we deduce that
$$\langle r \rangle = E(r)=E(r^p)= E(s) = \langle s \rangle.$$
Therefore, we must have $r= s^{\pm 1}$, hence $r=s$ by definition of $S$. 
\end{proof}

\begin{cor}\label{cor:GPimc}
Let $\Gamma$ be a finite graph and $\mathcal{G}$ a collection of groups indexed by $\Gamma$. If $\Gamma$ contains at least two vertices and is not a join, then $\Gamma \mathcal{G}$ has an IMC generating set.
\end{cor}

\begin{proof}
According to Proposition~\ref{prop:GPacyl}, Proposition~\ref{prop:AcylIMC} applies and yields the desired conclusion.
\end{proof}

\noindent
Let us record the following statement for future use:

\begin{lemma}\label{lem:ProductIMCinGP}
Let $\Gamma$ be a finite graph and $\mathcal{G}$ a collection of groups indexed by $\Gamma$. Let $Q$ be a maximal product subgroup of $\Gamma \mathcal{G}$ that is not contained in a conjugate of vertex-group given by an isolated vertex of $\Gamma$. Then $Q$ decomposes as $Q_1 \times \cdots \times Q_s$ where each $Q_i$ either is conjugate to a vertex-group or admits an IMC generating set. 
\end{lemma}

\begin{proof}
According to Lemma~\ref{lem:MaxProdInGP}, our maximal product subgroup $Q \leq \Gamma \mathcal{G}$ can written as $g \langle \Xi \rangle g^{-1}$ for some $g \in  \Gamma \mathcal{G}$ and some maximal join $\Xi \subset \Gamma$. Decompose $\Xi$ as $\Xi_1 \ast \cdots \ast \Xi_s$ where no $\Xi_i$ is a join. Accordingly, $Q$ decomposes as a product $Q_1 \times \cdots \times Q_s$ where $Q_i:= g \langle \Xi_i \rangle g^{-1}$ for every $1 \leq i \leq s$. For every $1 \leq i \leq s$, either $\Xi_i$ is reduced to a single vertex, in which case $Q_i$ is conjugate to a vertex-group; or $\Xi_i$ contains at least two vertices, in which case $Q_i$ admits an IMC generating set according to Corollary~\ref{cor:GPimc}.
\end{proof}

\subsection{Some rigidity}

\noindent
Our main motivation for the introduction of IMC generating sets is the following statement, which is inspired by \cite[Lemma~3.4]{MR4808711} and which will be fundamental in order to deduce Theorem~\ref{thm:NeverRAAG} from Theorem~\ref{thm:NotVirtRAAG}. 

\begin{lemma}\label{lem:IMCandProduct}
Let $H,K,A_1, \ldots, A_n$ be groups such that $H \times K$ is a finite-index subgroup of $A_1 \times \cdots \times A_n$. If $A_1, \ldots, A_n$ have almost stable centralisers and if $H$ is a non-cyclic group admitting an IMC generating set, then 
$$ H \leq \mathrm{VZ}(A_1) \times \cdots \times \mathrm{VZ}(A_{i-1}) \times A_i \times \mathrm{VZ}(A_{i+1}) \times \cdots \times \mathrm{VZ}(A_n)$$
for some index $1 \leq i \leq n$. 
\end{lemma}

\begin{proof}
Fix an IMC generating set $S \subset H$. Notice that, since $H$ is not cyclic, $S$ contains at least two elements. 

\medskip \noindent
Fix an $s \in S$. In $A:= A_1 \times \cdots \times A_n$, we can write $s=(s_1, \ldots, s_n)$. Let $k \geq 1$ be a sufficiently large integer so that $s_i^k \in H \cap K$ for every $1 \leq i \leq n$. If $s_i^k \in Z(A_i)$ for every $1 \leq i \leq n$, then $s^k$ belongs to the centre of $A$, and a fortiori of $H$. Then, $[r,s^k]=1$ for every $r \in S \backslash \{s\}$, contradicting the fact that $S$ is IMC. Thus, there exists some $1 \leq i \leq n$ such that $s_i^k \notin Z(A_i)$. It follows that
$$\begin{array}{lcl} C_{H \times K} (s) & \subset & C_{H \times K} (s^k) = (H \times K) \cap C_A(s^k) \\ \\ & = & (H \times K) \cap \left( C_{A_1}(s_1^k) \times \cdots \times C_{A_n} (s_n^k) \right) \\ \\ & \subsetneq & (H \times K) \cap \left( C_{A_1}(s_1^k) \times \cdots \times C_{A_{i-1}}(s_{i-1}^k) \times A_i \times C_{A_{i+1}}(s_{i+1}^k) \times \cdots \times C_{A_n}(s_n^k) \right) \\  \\ & \subsetneq & C_{H \times K} (s_1^k \cdots s_{i-1}^k s_{i+1}^k \cdots s_n^k) \end{array}$$
Since $S$ is IMC, it follows that $s_1^k \cdots s_{i-1}^k s_{i+1}^k \cdots s_n^k=1$, which amounts to saying that $s_j^k=1$ for every $j \neq i$, or equivalently that $s^k \in A_i$. 

\medskip \noindent
Notice that, for every $j \neq i$, $C_{A_j}(s_j)$ has finite index in $C_{A_j}(s_j^k) = C_{A_j}(1)= A_j$ since $A_j$ has almost stable centralisers, which amounts to saying that $s_j$ belongs to $\mathrm{VZ}(A_j)$. 

\medskip \noindent
So far, we have proved that $s$ belongs to 
$$\mathrm{VZ}(A_1) \times \cdots \times \mathrm{VZ}(A_{i-1}) \times A_i \times \mathrm{VZ}(A_{i+1}) \times \cdots \times \mathrm{VZ}(A_n)$$
and that $s^k \in A_i$. Given an $r \in S \backslash \{s\}$, we know similarly that there exist $\ell \geq 1$ and $1 \leq j \leq n$ such that $r$ belongs to 
$$\mathrm{VZ}(A_1) \times \cdots \times \mathrm{VZ}(A_{j-1}) \times A_i \times \mathrm{VZ}(A_{j+1}) \times \cdots \times \mathrm{VZ}(A_n)$$
and such that $r^\ell \in A_j$. If $i \neq j$, then clearly $[r^\ell, s^k]=1$, which is impossible as $S$ is IMC. Hence $i=j$. We conclude that $S$, and a fortiori $H$, is contained in 
$$\mathrm{VZ}(A_1) \times \cdots \times \mathrm{VZ}(A_{i-1}) \times A_i \times \mathrm{VZ}(A_{i+1}) \times \cdots \times \mathrm{VZ}(A_n),$$
as desired.
\end{proof}

\noindent
As a first consequence of Lemma~\ref{lem:IMCandProduct}, let us mention that:

\begin{cor}\label{cor:UniqueFactor}
Let $\Phi_1, \Phi_2$ be two finite graphs that contain at least two vertices and are not joins. If $A(\Psi_1) \times A(\Psi_2) = H_1 \times H_2$ for some non-trivial subgroups $H_1,H_2 \leq A(\Phi_1) \times A(\Phi_2)$, then $H_1= A(\Phi_1)$ and $H_2= A(\Phi_2)$ up to switching $H_1$ and $H_2$. 
\end{cor}

\begin{proof}
We know from Corollary~\ref{cor:GPimc} that $A(\Psi_1)$ and $A(\Psi_2)$ are not cyclic and admit an IMC generating set. Moreover, as a consequence of Corollary~\ref{cor:StableCent}, $H_1$ and $H_2$ have (almost) stable centralisers. Therefore, Lemma~\ref{lem:IMCandProduct} applies to the inclusion map $A(\Psi_1) \times A(\Psi_2) \hookrightarrow H_1 \times H_2$, proving that $A(\Psi_1)$ and $A(\Psi_2)$ are contained in $H_1$ or $H_2$. Since $H_1$ and $H_2$ are non-trivial, clearly $A(\Psi_1)$ and $A(\Psi_2)$ cannot be both contained in either $H_1$ or $H_2$. Up to switching $H_1$ and $H_2$, say that $A(\Psi_1) \leq H_1$ and $A(\Psi_2) \leq H_2$. From the equality $A(\Psi_1) \times A(\Psi_2) = H_1 \times H_2$, we conclude that $H_1= A(\Phi_1)$ and $H_2= A(\Phi_2)$.
\end{proof}

\section{Flat braids on seven strands}

\noindent
This section is dedicated to the proof of the following statement, which will be the key in order to deduce that most pure flat braid groups are not virtually right-angled Artin groups.

\begin{thm}\label{thm:NotVirtRAAG}
The group $\mathrm{PFB}_7$ is not virtually a right-angled Artin group.
\end{thm}

\noindent
We denote by $\sigma_1, \ldots, \sigma_6$ the usual generators of $\mathrm{FB}_7$, i.e.\ each $\sigma_i$ is an elementary twist of the $i$th and $(i+1)$st strands. Equivalently, when thinking of $\mathrm{FB}_7$ as the right-angled Coxeter group $C(P_5^\mathrm{opp})$, $\sigma_1, \ldots, \sigma_6$ correspond to the generators given by the successive vertices along the path $P_5$. 

\medskip \noindent
From the presentation
$$\langle \sigma_1, \ldots, \sigma_6 \mid \sigma_i^2=1 \ (1 \leq i \leq 6), \ [\sigma_i,\sigma_j]=1 \ (|i-j| \geq 2) \rangle$$
of $\mathrm{FB}_7$, one easily verifies that $\mathrm{FB}_7$ decomposes as the following graph of groups:

\begin{center}
\includegraphics[width=0.8\linewidth]{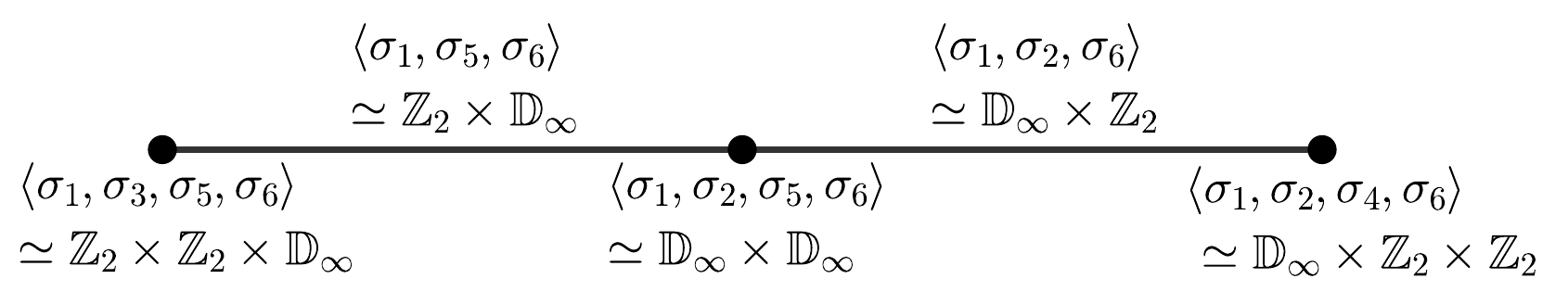}
\end{center}

\noindent
We refer to this decomposition of $\mathrm{FB}_7$ as its \emph{tubular decomposition}. The key point is that vertex-groups are virtually $\mathbb{Z}$ or $\mathbb{Z}^2$ and that edge-groups are virtually $\mathbb{Z}$. 

\medskip \noindent
We use this tubular decomposition in order to find restrictions on the possible subgroups of $\mathrm{FB}_7$. This will allow us to show that the only right-angled Artin groups that are subgroups of $\mathrm{PFB}_7$ are of the form $A(\Gamma)$ where $\Gamma$ is a forest (see Corollary~\ref{cor:SubRAAG}). 

\begin{lemma}\label{lem:NoSubgroup}
The group $\mathrm{FB}_7$ does not contain a subgroup isomorphic to $\mathbb{Z}^3$, to $\mathbb{F}_2 \times \mathbb{F}_2$, or to the fundamental group of a closed surface of genus $\geq 2$. 
\end{lemma}

\begin{proof}
As a consequence of its tubular decomposition, $\mathrm{FB}_7$ acts on a tree $T$ with virtually $\mathbb{Z}$ or $\mathbb{Z}^2$ vertex-stabilisers and with virtually $\mathbb{Z}$ edge-stabilisers. If $g \in \mathrm{FB}_7$ induces a loxodromic isometry on $T$, then its centraliser $C(g)$ in $\mathrm{FB}_7$, which must stabilise the axis of $\gamma$ of $g$ in $T$, is necessarily (virtually cyclic)-by-$\mathbb{Z}$. (Indeed, the action of $C(g)$ on $\gamma$ by translation induces an epimorphism to $\mathbb{Z}$ whose kernel fixes $\gamma$ pointwise, and consequently must be virtually cyclic since edge-groups are virtually $\mathbb{Z}$.) 

\medskip \noindent
Since every element of $\mathbb{Z}^3$ has centraliser $\mathbb{Z}^3$, it follows that, if $\mathbb{Z}^3$ is a subgroup of $\mathrm{FB}_7$, then it cannot contain a loxodromic isometry. Thus, it must be elliptic in $T$, which is impossible since vertex-stabilisers are virtually $\mathbb{Z}$ or $\mathbb{Z}^2$. Therefore, $\mathbb{Z}^3$ cannot be a subgroup of $\mathrm{FB}_7$.

\medskip \noindent
If $\mathbb{F}_2 \times \mathbb{F}_2$ is a subgroup of $\mathrm{FB}_7$, then an $\mathbb{F}_2$-factor cannot be elliptic in $T$, since vertex-stabilisers are virtually $\mathbb{Z}$ or $\mathbb{Z}^2$, so it must contain a loxodromic isometry. But the centraliser of an element of an $\mathbb{F}_2$-factor always contains a non-abelian subgroup, contracting the previous observation. Therefore, $\mathbb{F}_2 \times \mathbb{F}_2$ cannot be a subgroup of $\mathrm{FB}_7$.

\medskip \noindent
Finally, it remains to verify that $\mathrm{FB}_7$ does not contain a subgroup isomorphic to the fundamental group of a closed surface of genus $\geq 2$. We will prove more generally that:

\begin{claim}
Let $G$ be a one-ended hyperbolic group. Then $G$ is not isomorphic to a subgroup of $\mathrm{FB}_7$.
\end{claim}

\noindent
What we need to know about hyperbolic groups is that, for every infinite-order element $g \in G$, there exists a unique maximal virtually cyclic subgroup $E(g)$ containing $\langle g \rangle$, which we will refer to as the \emph{elementary closure}. (Geometrically, given a quasi-axis $\gamma$ of $g$ in $G$, $E(g)$ corresponds to the subgroup given by the elements $h \in G$ such that the Hausdorff distance between $\gamma$ and $h \gamma$ is finite. Or equivalently, to the stabiliser of the pair of points at infinity of $\gamma$.) Assume for contradiction that $G$ is isomorphic to a subgroup of $\mathrm{FB}_7$. Fix a minimal $G$-invariant subtree in the Bass-Serre tree associated to the tubular decomposition of $\mathrm{FB}_7$. Because $G$ is one-ended, edge-groups must be virtually $\mathbb{Z}$; and, because $G$ does not contain $\mathbb{Z}^2$, vertex-subgroups must be virtually $\mathbb{Z}$ as well. Thus, vertex-groups are pairwise commensurable in $G$. This implies that they all have the same elementary closure $E$. Then, $E$ yields a normal virtually $\mathbb{Z}$ subgroup of $G$. The only possibility is that $G$ is virtually $\mathbb{Z}$ itself, which is impossible as $G$ is supposed to be one-ended.
\end{proof}

\begin{cor}\label{cor:SubRAAG}
If a right-angled Artin group $A(\Gamma)$ embeds into $\mathrm{FB}_7$, then $\Gamma$ is a forest. Moreover, if $\Gamma$ is disconnected, then $A(\Gamma)$ has infinite index in $\mathrm{FB}_7$. 
\end{cor}

\begin{proof}
If $\Gamma$ contains an induced cycle of length three (resp.\ four, at least five), then $A(\Gamma)$ contains a subgroup isomorphic to $\mathbb{Z}^3$ (resp.\ $\mathbb{F}_2 \times \mathbb{F}_2$, the fundamental group of a closed surface of genus $\geq 2$ (see for instance \cite{MR952322})), which prevents $\mathrm{FB}_7$ from containing $A(\Gamma)$ according to Lemma~\ref{lem:NoSubgroup}. Therefore, $\Gamma$ must be a forest. Moreover, if $\Gamma$ is disconnected, then $A(\Gamma)$ splits as a free product of two infinite groups, which prevents $\mathrm{FB}_7$ from containing $A(\Gamma)$ as a finite-index subgroup since $\mathrm{FB}_7$ is one-ended (which amounts to saying, when thinking of $\mathrm{FB}_7$ as the right-angled Coxeter group $C(P_5^\mathrm{opp})$, that no complete subgraph separates $P_5^\mathrm{opp}$).  
\end{proof}

\noindent
In order to prove Theorem~\ref{thm:NotVirtRAAG}, it remains to distinguish $\mathrm{PFB}_7$ from right-angled Artin groups defined by trees. For this purpose, we introduce a specific subgroup.

\begin{definition}
Let $G$ be a group. An element $g \in G$ is \emph{thick} if its centraliser is not virtually abelian. The \emph{thick subgroup} $\mathrm{Thick}(G)$ is the subgroup of $G$ generated by the centralisers of all its thick elements.
\end{definition}

\noindent
It is worth noticing that, since conjugates of thick elements are thick themselves, thick subgroups are always normal. The key observation is that thick subgroups are not proper for right-angled Artin groups defined by trees while the thick subgroup of $\mathrm{PFB}_7$ is rather small (in particular, it has infinite index). This is the content of the next two lemmas.

\begin{lemma}\label{lem:ModThickRAAG}
For every tree $\Gamma$ with at least three vertices, $\mathrm{Thick}(A(\Gamma)) = A(\Gamma)$. 
\end{lemma}

\begin{proof}
For every vertex $u \in \Gamma$ of degree $\geq 2$, the centraliser of the corresponding generator of $A(\Gamma)$ is $\langle \mathrm{link}(u) \rangle$, which is a free group of rank $\geq 2$. Therefore, such a generator is a thick element. Since a leaf of $\Gamma$ must be adjacent to some vertex of degree $\geq 2$, a generator of $A(\Gamma)$ that is not thick must belong to the centraliser of a thick centraliser. Therefore, every generator of $A(\Gamma)$ belongs to the thick subgroup, hence the desired equality.
\end{proof}

\begin{lemma}\label{lem:ModThickFB}
The quotient $\mathrm{FB}_7 / \mathrm{Thick}( \mathrm{PFB}_7)$ is isomorphic to the Coxeter group $C(\Gamma)$ where $\Gamma$ is the labelled graph defined as follows:
\begin{itemize}
	\item the underlying graph of $\Gamma$ is a complete graph with six vertices $x_1, \ldots, x_6$;
	\item the edge connecting $x_i$ and $x_j$ is labelled by $2$ whenever $|i-j| \geq 2$;
	\item the edge connecting $x_i$ and $x_{i+1}$ is labelled by $3$ if $i \neq 3$ and $\infty$ otherwise.
\end{itemize}
In particular, it is infinite.
\end{lemma}

\begin{proof}
Let $\pi : \mathrm{FB}_7 \to C(\Gamma)$ denote the morphism that sends $\sigma_i$ to $x_i$ for every $1 \leq i \leq 6$. Our goal is to prove that $\mathrm{ker}(\pi)= \mathrm{Thick}(\mathrm{PFB}_7)$, which will conclude the proof of our lemma. First, let us verify that $\mathrm{Thick}(\mathrm{PFB}_7)$ is contained in $\mathrm{ker}(\pi)$. 

\begin{claim}\label{claim:ThickFB}
An element of $\mathrm{FB}_7$ is thick if and only if it is conjugate to a non-trivial power of $\sigma_1\sigma_2$ or $\sigma_5\sigma_6$.
\end{claim}

\noindent
We think of $\mathrm{FB}_7$ as the right-angled Coxeter group $C(P_5^\mathrm{opp})$. Let $g \in \mathrm{FB}_7$ be a thick element. Up to conjugating $g$, we can assume that $g$ is graphically cyclically reduced. Because $g$ has infinite order, necessarily $\mathrm{supp}(g)$ is not complete. And, because the centraliser of $g$ is not virtually abelian, it follows from Proposition~\ref{prop:CentralisersGP} that $\mathrm{link}(\mathrm{supp}(g))$ contains at least two non-adjacent vertices but is not just a pair of two non-adjacent vertices. In $P_5^\mathrm{opp}$, there are only two possibilities: either $\mathrm{supp}(g)= \{ \sigma_1, \sigma_2\}$ and $\mathrm{link}(\mathrm{supp}(g))=\{ \sigma_4, \sigma_5, \sigma_6\}$, or $\mathrm{supp}(g)= \{ \sigma_5, \sigma_6\}$ and $\mathrm{link}(\mathrm{supp}(g))= \{ \sigma_1, \sigma_2, \sigma_3\}$. In the first case, $g$ is a non-trivial power of $\sigma_1 \sigma_2$; and, in the second case, $g$ is a non-trivial power of $\sigma_5 \sigma_6$. Conversely, it follows from Proposition~\ref{prop:CentralisersGP} that the centraliser of a non-trivial power $(\sigma_1 \sigma_2)^k$ (resp.\ $(\sigma_5 \sigma_6)^k$) is $\langle \sigma_1 \sigma_2 \rangle \times \langle \sigma_4 , \sigma_5, \sigma_6 \rangle$ (resp.\ $\langle \sigma_1, \sigma_2, \sigma_3 \rangle \times \langle \sigma_5 \sigma_6 \rangle$), which is isomorphic to $\mathbb{Z} \times (\mathbb{Z}_2 \ast (\mathbb{Z}_2\times \mathbb{Z}_2))$. Thus, the centraliser of $\sigma_1 \sigma_2$ (resp.\ $\sigma_5 \sigma_6$) is indeed not virtually abelian. This concludes the proof of Claim~\ref{claim:ThickFB}. 

\medskip \noindent
Since an element of $\mathrm{PFB}_7$ is thick in $\mathrm{PFB}_7$ if and only if it is thick in $\mathrm{FB}_7$, it follows from Claim~\ref{claim:ThickFB} that an element of $\mathrm{PFB}_7$ is thick if and only if it is conjugate to a non-trivial power of $(\sigma_1\sigma_2)^3$ or $(\sigma_5 \sigma_6)^3$. We deduce from Proposition~\ref{prop:CentralisersGP} that the centraliser of $(\sigma_1\sigma_2)^3$ in $\mathrm{FB}_7$ is $\langle \sigma_1\sigma_2 \rangle \times \langle \sigma_4, \sigma_5, \sigma_6 \rangle$, so the centraliser of $(\sigma_1\sigma_2)^3$ in $\mathrm{PFB}_7$ is
$$\langle (\sigma_1 \sigma_2)^3 \rangle \times \left( \mathrm{PFB}_7 \cap \langle \sigma_4,\sigma_5, \sigma_6 \rangle \right).$$
By noticing that $\pi$ restricts on $\langle \sigma_1, \sigma_2, \sigma_3 \rangle \simeq \mathrm{FB}_4$ (resp.\ $\langle \sigma_4, \sigma_5, \sigma_6 \rangle \simeq \mathrm{FB}_4$) to the canonical map to the permutation group $\langle x_1,x_2,x_3 \rangle \simeq \mathrm{Sym}(4)$ (resp.\ $\langle x_4,x_5,x_6 \rangle \simeq \mathrm{Sym}(4)$), it follows that the centraliser above is contained in the kernel of $\pi$. Symmetrically, we show that the centraliser of $(\sigma_5\sigma_6)^3$ in $\mathrm{PFB}_7$ is contained in $\mathrm{ker}(\pi)$. Thus, we have proved that $\mathrm{ker}(\pi)$ contains the centraliser in $\mathrm{PFB}_7$ of every thick element of $\mathrm{PFB}_7$. In other words, $\mathrm{Thick}(\mathrm{PFB}_7) \leq \mathrm{ker}(\pi)$, as desired.

\medskip \noindent
By comparing the Coxeter presentations of $\mathrm{FB}_7$ and $C(\Gamma)$, it is clear that $\mathrm{ker}(\pi)$ coincides with the normal closure in $\mathrm{FB}_7$ of $\{ (\sigma_1\sigma_2)^3, (\sigma_2\sigma_3)^3, (\sigma_4\sigma_5)^3, (\sigma_5\sigma_6)^3 \}$. Because $\mathrm{Thick}(\mathrm{PFB}_7)$ is a normal subgroup of $\mathrm{FB}_7$, as a consequence of Fact~\ref{fact:SubThickNormal} below, it suffices to notice that $(\sigma_1\sigma_2)^3, (\sigma_2\sigma_3)^3, (\sigma_4\sigma_5)^3, (\sigma_5\sigma_6)^3$ all belong to $\mathrm{Thick}(\mathrm{PFB}_7)$ in order to conclude that $\mathrm{ker}(\pi)$ is contained in $\mathrm{Thick}(\mathrm{PFB}_7)$. But we already know that $(\sigma_1\sigma_2)^3$ and $(\sigma_5\sigma_6)^3$ are thick elements of $\mathrm{PFB}_7$, and $(\sigma_2\sigma_3)^3$ (resp.\ $(\sigma_4 \sigma_5)^3$) belongs to the centraliser of $(\sigma_5\sigma_6)^3$ (resp.\ of $(\sigma_1\sigma_2)^3$). 

\begin{fact}\label{fact:SubThickNormal}
Let $G$ be a group. For every normal subgroup $H \lhd G$, $\mathrm{Thick}(H)$ is a normal subgroup of $G$. 
\end{fact}

\noindent
The action of $G$ on $H$ by conjugation permutes the thick elements of $H$ and sends centralisers to centralisers. Therefore, $G$ permutes the centralisers of the thick elements of $H$, proving that $\mathrm{Thick}(H)$ is stabilised by conjugation, or equivalently that $\mathrm{Thick}(H)$ is a normal subgroup of $G$. 
\end{proof}

\begin{proof}[Proof of Theorem~\ref{thm:NotVirtRAAG}.]
Assume for contradiction that $\mathrm{PFB}_7$ contains a right-angled Artin group $A(\Gamma)$ as a finite-index subgroup. It follows from Corollary~\ref{cor:SubRAAG} that $\Gamma$ must be a tree (with at least three vertices since $\mathrm{FB}_7$ is not virtually abelian), hence $\mathrm{Thick}(A(\Gamma))= A(\Gamma)$ according to Lemma~\ref{lem:ModThickRAAG}. But we clearly have $\mathrm{Thick}(A(\Gamma)) \leq \mathrm{Thick}(\mathrm{PFB}_7)$, hence
$$|\mathrm{PFB}_7 / \mathrm{Thick}(\mathrm{PFB}_7)| \leq | \mathrm{PFB}_7 / A(\Gamma)| < \infty.$$
This contradicts Lemma~\ref{lem:ModThickFB}, which implies that $\mathrm{PFB}_7 / \mathrm{Thick}(\mathrm{PFB}_7)$ is infinite. 
\end{proof}

\section{Pure flat braid groups are not right-angled Artin groups}

\noindent
In this section, we prove the main result of this article, namely:

\begin{thm}\label{thm:NeverRAAG}
For every $n=7$ or $\geq 11$, $\mathrm{PFB}_n$ is not virtually a right-angled Artin group.
\end{thm}

\noindent
We start by stating and proving a general criterion that allows us to show that, under some assumptions, if a group $G_1$ can be realised as a finite-index subgroup in a group $G_2$, then every factor of a maximal product subgroup of $G_2$ contains as a finite-index subgroup a factor of a maximal product subgroup of $G_1$. 

\begin{prop}\label{prop:MorphismProduct}
Let $G_1$ and $G_2$ be two torsion-free groups. Assume that:
\begin{itemize}
	\item[(i)] In both $G_1$ and $G_2$, a subgroup commensurable to a product of two infinite groups is contained in a maximal product subgroup.
	\item[(ii)] In $G_1$, every maximal product subgroup $Q$ decomposes as $Q_1 \times \cdots \times Q_s$ where each $Q_i$ either is infinite cyclic or admits an IMC generating set.
	\item[(iii)] In $G_2$, if two maximal product subgroups $P_1$ and $P_2$ are such that $P_1 \cap P_2$ has finite index in $P_1$, then $P_1=P_2$.
	\item[(iv)] $G_2$ has almost stable centralisers.
\end{itemize}
If $G_1$ is a finite-index subgroup of $G_2$ and if $P:= P_1 \times \cdots \times P_s$ is a maximal product subgroup of $G_2$ such that $\mathrm{VZ}(P)=\{1\}$ and such that no $P_i$ is virtually a product of two infinite groups, then there exists a maximal product subgroup $R:=R_1 \times \cdots \times R_s$ of $G_1$ such that each $R_i$ is a finite-index subgroup of $P_i$.
\end{prop}

\begin{proof}
Let $P = P_1 \times \cdots \times P_s$ be a maximal product subgroup of $G_2$ such that $\mathrm{VZ}(P)=\{1\}$ and such that no $P_i$ is virtually a product of two infinite groups. It follows from $(i)$ that $P \cap G_1$ is contained in some maximal product subgroup $R \leq G_1$. Similarly, $R$ must be contained in some maximal product subgroup $P^+$ of $G_2$. Since $P \cap G_1$ has finite index in $P$, it follows from $(iii)$ that $P^+$ actually coincides with $P$. Let $R= R_1 \times \cdots \times R_m$ denote the decomposition of $R$ given by $(ii)$. Notice that, since $\mathrm{VZ}(P)$ is trivial, necessarily $\mathrm{VZ}(R)$ must be trivial as well, which implies that no $R_i$ is (virtually) cyclic. 

\medskip \noindent
Notice that $P \cap G_1 \leq R \leq P$ and that $P \cap G_1$ has finite index in $P$, so $R$ must have finite index in $P$. Applying Lemma~\ref{lem:IMCandProduct}, which is possible thanks to $(ii)$ and $(iv)$, we deduce that, for every $1 \leq i \leq m$, there exists $1 \leq \sigma(i) \leq s$ such that $R_i \leq P_{\sigma(i)}$. Notice that
$$P/R \equiv \prod\limits_{i=1}^r \left( P_i / \prod\limits_{j \in \sigma^{-1}(i)} R_j \right),$$
which must be finite. Consequently, $\prod_{j \in \sigma^{-1}(i)} R_j$ must have finite index in $P_i$ for every $1 \leq i \leq r$. But we know by assumption that no $P_i$ is virtually a product, so $\sigma$ must be bijective. Thus, up to reordering the factors of $P$, we have proved that $m=s$ and that $R_i$ has finite index in $P_i$ for every $1 \leq i \leq s$, as desired.
\end{proof}

\noindent
Thanks to Lemmas~\ref{lem:GPcontainedInProduct},~\ref{lem:ProductIMCinGP},~\ref{lem:BigInterProd}, and~\ref{lem:GPstableCent}, we have a good understanding of when a graph product satisfies the assumptions of Proposition~\ref{prop:MorphismProduct}. This applies in particular to flat braid groups, from which it is not difficult to deduce similar statements for their pure subgroups:

\begin{lemma}\label{lem:ForPFB}
For every $n \geq 2$, the following assertions hold:
\begin{itemize}
	\item[(i)] A subgroup of $\mathrm{PFB}_n$ commensurable to a product of two infinite groups is contained in a maximal product subgroup.
	\item[(ii)] In $\mathrm{PFB}_n$, if two maximal product subgroups $P_1$ and $P_2$ are such that $P_1 \cap P_2$ has finite index in $P_1$, then $P_1=P_2$.
	\item[(iii)] $\mathrm{PFB}_n$ has almost stable centralisers.
	\item[(iv)] The maximal product subgroups of $\mathrm{PFB}_n$ are the conjugates of $$\mathrm{PFB}_{i} \times \mathrm{PFB}_{n-i}:= \left( \mathrm{PFB}_n \cap \langle \sigma_1, \ldots, \sigma_{i-1} \rangle \right) \times \left( \mathrm{PFB}_n \cap \langle \sigma_{i+1}, \ldots, \sigma_{n-1} \right)$$ for $3 \leq i \leq n-3$.
\end{itemize}
\end{lemma}

\begin{proof}
Let $P$ be a product in $\mathrm{PFB}_n$. According to Lemma~\ref{lem:GPcontainedInProduct}, $P$ is contained in a maximal product subgroup $Q$ of $\mathrm{FB}_n$. Thinking of $\mathrm{FB}_n$ as a right-angled Coxeter group, we deduce from Lemma~\ref{lem:MaxProdInGP} that $Q$ is conjugate to $\langle \sigma_1, \ldots, \sigma_{i-1} \rangle \times \langle \sigma_{i+1}, \ldots, \sigma_{n-1}\rangle$ for some $3 \leq i \leq n-3$. Thus, $P$ is contained in a conjugate of $\mathrm{PFB}_{i} \times \mathrm{PFB}_{n-i}$. This proves $(iv)$. Moreover, our argument shows the following assertion, which we will use in order to prove the rest of the lemma:

\begin{fact}\label{fact:MaxSub}
The maximal product subgroups of $\mathrm{PFB}_n$ are the intersection with $\mathrm{PFB}_n$ of the maximal product subgroups of $\mathrm{FB}_n$. 
\end{fact}

 \noindent
If $H \leq \mathrm{PFB}_n$ is commensurable to a product of two infinite groups, then we know from Lemma~\ref{lem:GPcontainedInProduct} that $H$ is contained in a maximal product subgroup $Q$ of $\mathrm{FB}_n$. Thus, $H$ is contained in $Q \cap \mathrm{PFB}_n$, which is a maximal product subgroup of $\mathrm{PFB}_n$ according to Fact~\ref{fact:MaxSub}. This proves $(i)$.

\medskip \noindent
Let $P_1$ and $P_2$ be two maximal product subgroups of $\mathrm{PFB}_n$ such that $P_1 \cap P_2$ has finite index in $P_1$. According to Lemma~\ref{lem:GPcontainedInProduct}, $P_1$ (resp.\ $P_2$) is contained in a maximal product subgroup $P_1^+$ (resp.\ $P_2^+$) of $\mathrm{FB}_n$. Notice that, as a consequence of Fact~\ref{fact:MaxSub}, $P_1^+ \cap \mathrm{PFB}_n$ (resp.\ $P_2^+ \cap \mathrm{PFB}_n$) is a maximal product subgroup of $\mathrm{PFB}_n$. Since it contains $P_1$ (resp.\ $P_2$), necessarily $P_1= P_1^+ \cap \mathrm{PFB}_n$ (resp.\ $P_2= P_2^+ \cap \mathrm{PFB}_n$). Since $P_1^+ \cap P_2^+$ contains $P_1 \cap P_2$, and since the latter has finite index in $P_1$, and a fortiori in $P_1^+$, we deduce from Lemma~\ref{lem:BigInterProd} that $P_1^+ = P_2^+$. Hence 
$$P_1 = P_1^+ \cap  \mathrm{PFB}_n = P_2^+ \cap \mathrm{PFB}_n = P_2,$$
proving the assertion $(ii)$.

\medskip \noindent
Finally, the assertion $(iii)$ follows from Corollary~\ref{cor:StableCent}, since pure flat braid groups embed into right-angled Artin groups (as a consequence of \cite[Example~5.40]{MR4071367} (see also \cite{Farley}) and \cite[Theorem~1.2]{MR3868219}).
\end{proof}

\begin{proof}[Proof of Theorem~\ref{thm:NeverRAAG}.] 
For $n=7$, the desired conclusion is given by Theorem~\ref{thm:NotVirtRAAG}. From now on, assume that $n \geq 9$. Notice that, according to Lemmas~\ref{lem:GPcontainedInProduct},~\ref{lem:ProductIMCinGP}, and~\ref{lem:ForPFB}, the assumptions $(i)-(iv)$ of Proposition~\ref{prop:MorphismProduct} are satisfied for $G_1$ a right-angled Artin group and $G_2$ a pure flat braid group.

\medskip \noindent
Assume for contradiction that $\mathrm{PFB}_n$ contains a finite-index subgroup that is a right-angled Artin group $A(\Gamma)$. According to Lemma~\ref{lem:ForPFB}, $P:= \mathrm{PFB}_7 \times \mathrm{PFB}_{n-7}$ is a maximal product subgroup of $\mathrm{PFB}_n$. We know from Corollaries~\ref{cor:FBVZ} and~\ref{cor:PFBnotProd} that $\mathrm{PFB}_k$ is not virtually a product of two infinite groups and has a trivial virtual centre whenever $k \geq 4$. Consequently, Proposition~\ref{prop:MorphismProduct} applies and shows that $A(\Gamma)$ contains a maximal product subgroup $R:=R_1 \times R_2$ such that $R_1$ (resp.\ $R_2$) has finite index in $\mathrm{PFB}_7$ (resp.\ $\mathrm{PFB}_{n-7}$. Since $R_1$ must be a right-angled Artin group according to Corollary~\ref{cor:UniqueFactor}, we deduce that $\mathrm{PFB}_7$ is virtually a right-angled Artin group, contradicting Theorem~\ref{thm:NotVirtRAAG}. 
\end{proof}

\section{An instance of commensurability}

\noindent
In contrast with Theorem~\ref{thm:NeverRAAG}, in this section we prove the following observation:

\begin{thm}\label{thm:Commensurable}
The flat braid group $\mathrm{FB}_7$ is commensurable to the right-angled Artin group $A(P_4)$.
\end{thm}

\noindent
First, we show in Section~\ref{section:Lampraag} that $\mathrm{FB}_7$ is commensurable to a \emph{lampraag} $A(\mathbb{L}) \rtimes \mathbb{Z}$, thanks to some (quasi-)median geometry explained in Section~\ref{section:QM}. A key observation is that both $A(\mathbb{L}) \rtimes \mathbb{Z}$ and $A(P_4)$ can be described as fundamental groups of compact $3$-manifolds. Therefore, in order to deduce that these two groups are commensurable, it suffices to construct a common finite-sheeted cover for the two corresponding manifolds. This is what we do in Section~\ref{section:Flip}.

\subsection{Some (quasi-)median geometry}\label{section:QM}

\noindent
In the next section, we will use some quasi-median geometry of graph products, as introduced in \cite{QM}. In this section, we recall the few definitions and results that we need.

\begin{thm}[\cite{QM}]\label{thm:QM}
For every graph $\Gamma$ and every collection of groups $\mathcal{G}$ indexed by $\Gamma$, the Cayley graph 
$$\mathrm{QM}(\Gamma, \mathcal{G}):= \mathrm{Cayl} \left( \Gamma \mathcal{G} , \bigcup\limits_{G \in \mathcal{G}} G \backslash \{1\} \right).$$
is a quasi-median graph.
\end{thm}

\noindent
Quasi-median graphs can be defined as retracts of Hamming graphs (i.e.\ products of complete graphs). There are many alternative characterisations of quasi-median graphs (see for instance \cite{MR1297190}), but this definition is rather simple and it highlights the connection with median graphs (also known as one-skeletons of CAT(0) cube complexes\footnote{See \cite{MedianNotCCC} and references therein.}), which can be defined as retracts of hypercubes. It is also worth mentioning that median graphs coincide with triangle-free quasi-median graphs. As a consequence, Theorem~\ref{thm:QM} implies the well-known observation:

\begin{cor}\label{cor:RACGmedian}
For every graph $\Gamma$, the Cayley graph $\mathrm{Cayl}(C(\Gamma), \Gamma)$ of the right-angled Coxeter group $C(\Gamma)$ is a median graph. 
\end{cor}

\noindent
\emph{Hyperplanes} are the key objects in order to understand the geometry of (quasi-)median graphs. They are equivalence classes of edges with respect to the reflexive-transitive closure of the relation that identifies two edges whenever they belong to a common $3$-cycle or whenever they are opposite edges in a $4$-cycle. See Figure~\ref{Hyp}. Hyperplanes in quasi-median graphs of graph products are described by \cite[Lemma~8.9 and Corollary~8.10]{QM}:
\begin{figure}
\begin{center}
\includegraphics[trim= 0 16.5cm 10cm 0, clip, width=0.5\linewidth]{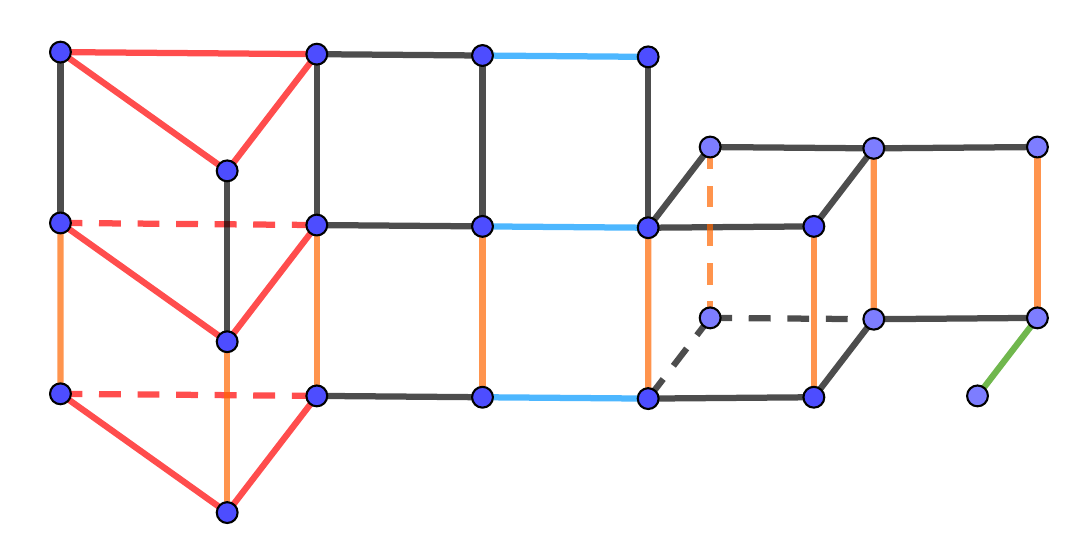}
\caption{Some hyperplanes in a quasi-median graph.}
\label{Hyp}
\end{center}
\end{figure}

\begin{lemma}\label{lem:QMhyp}
Let $\Gamma$ be a graph and $\mathcal{G}$ a collection of groups indexed by $\Gamma$. Fix a vertex $u \in \Gamma$ and let $J_u$ denote the hyperplane of $\mathrm{QM}(\Gamma, \mathcal{G})$ containing the edges of the clique $\langle u \rangle$. An edge $e$ belongs to $J_u$ if and only if $e = \{ g,g \ell\}$ for some $g \in \langle \mathrm{link}(u) \rangle$ and $\ell \in \langle u \rangle \backslash \{1\}$.  As a consequence, the stabiliser of $J_u$ in $\Gamma \mathcal{G}$ is $\langle \mathrm{star}(u) \rangle$. 
\end{lemma}

\noindent
As a consequence of our lemma, all the edges of a given hyperplane of $\mathrm{QM}(\Gamma, \mathcal{G})$ are naturally labelled by the same vertex of $\Gamma$. This labelling may quite useful in practice. For instance:

\begin{lemma}\label{lem:AdjLabels}
Let $\Gamma$ be a graph and $\mathcal{G}$ a collection of groups indexed by $\Gamma$. Two transverse hyperplanes in $\mathrm{QM}(\Gamma, \mathcal{G})$ must be labelled by adjacent vertices of $\Gamma$. 
\end{lemma}

\noindent
In order to state our last preliminary lemma, we need to introduce a couple of preliminary definitions.

\begin{definition}
Let $G$ be a group acting on a quasi-median graph $X$. The \emph{rotative-stabiliser} $\mathrm{stab}_\circlearrowright(J)$ of a hyperplane $J$ is the subgroup of $\mathrm{stab}(J)$ that stabilises each maximal complete subgraph of $J$. 
\end{definition}

\noindent
In the sequel, given a group $G$ acting on a median graph $X$ and a collection of hyperplanes $\mathcal{J}$, we denote by $\mathrm{Rot}(\mathcal{J})$ the subgroup $\langle \mathrm{stab}_\circlearrowright(J) \mid J \in \mathcal{J} \rangle$ of $G$. 

\begin{definition}
Let $X$ be a quasi-median graph and $\mathcal{J}$ a collection of hyperplanes. The \emph{crossing graph} of $\mathcal{J}$ is the graph whose vertices are the hyperplanes of $\mathcal{J}$ and whose edges connect two hyperplanes whenever they are transverse. 
\end{definition}

\noindent
Recall that two hyperplanes are \emph{transverse} whenever they cover some $4$-cycle. We also say that a hyperplane $J$ is \emph{tangent} to a subgraph $Y$ whenever $J$ contains an edge not in $Y$ but with a vertex in $Y$. 

\begin{prop}\label{prop:PingPong}
Let $G$ be a group acting on a quasi-median graph $X$. Fix a gated subgraph $Y \subset X$ and let $\mathcal{J}$ denote the collection of the hyperplanes of $X$ tangent to $Y$. Assume that the following conditions holds:
\begin{itemize}
	\item for every $J \in \mathcal{J}$, $\mathrm{stab}_\circlearrowright(J)$ acts vertex-freely on $X$;
	\item for all transverse $J,H \in \mathcal{J}$, every element of $\mathrm{stab}_\circlearrowright(J)$ commutes with every element of $\mathrm{stab}_\circlearrowright(H)$.
\end{itemize}
Let $\Delta$ denote the crossing graph of $\mathcal{J}$ and let $\mathcal{G} = \{ \mathrm{stab}_\circlearrowright(J) \mid J \in \mathcal{J}\}$. Then
$$\langle \mathrm{Rot}(\mathcal{J}), \mathrm{stab}(Y) \rangle = \mathrm{Rot}(\mathcal{J}) \rtimes \mathrm{stab}(Y)$$
and the map $\Delta \mathcal{G} \to \mathrm{Rot}(\mathcal{J})$ that restricts on each vertex-group $\mathrm{stab}_\circlearrowright(J)$ to the identity is an isomorphism. 
\end{prop}

\noindent
The proposition is a rather straightforward consequence of the Ping-Pong Lemma \cite[Proposition~8.44]{QM}. It is also a consequence of \cite[Theorem~10.54]{QM}. In order to keep the paper short, we refer the reader to \cite{QM} for more details and we only mention that Proposition~\ref{prop:PingPong} applies to subgroups $G \leq \Gamma \mathcal{G}$ of graph products $\Gamma \mathcal{G}$, to quasi-median graphs $X= \mathrm{QM}(\Gamma, \mathcal{G})$, and to subgraphs of the form $\langle \Lambda \rangle$ where $\Lambda \subset \Gamma$ as justified by \cite[Lemma~8.46]{QM} and \cite[Corollary~6.6]{Mediangle}.

\subsection{Lampraag over $\mathbb{Z}$}\label{section:Lampraag}

\noindent
In this section, we introduce a family of groups we call \emph{lampraags}, in analogy with lamplighter groups. They are particular examples of graph-wreath products \cite{MR3393469}, and more generally of halo products \cite{Halo}. The lampraag over $\mathbb{Z}$ will provide a convenient model for $\mathrm{FB}_7$ up to commensurability.

\begin{definition}
Let $G$ be a group and $S \subset G$ a subset. The \emph{lampraag over $(G,S)$} is the semidirect product 
$$A( \mathrm{Cayl}(G,S) ) \rtimes G$$
where $G$ permutes the generators of $A(\mathrm{Cayl}(G,S))$ according to its action by left-multiplication on $\mathrm{Cayl}(G,S)$. 
\end{definition}

\noindent
In the definition, we do not require $S$ to be a generating set. For instance, taking $S= \emptyset$ is allowed, in which case $A(\mathrm{Cayl}(G,\emptyset)) \rtimes G \simeq \mathbb{Z} \ast G$. In practice, however, we will be mainly interested in the case where $\mathrm{Cayl}(G,S)$ is connected, i.e.\ when $S$ is a generating~set. In fact, in the sequel, we will focus on the lampraag over $\mathbb{Z}$, where $\mathbb{Z}$ is endowed with its canonical generating set $\{1\}$. For convenience, we will denote by $\mathbb{L}$ the Cayley graph of $\mathbb{Z}$ with respect to $\{1\}$, i.e.\ the bi-infinite line; and by $A(\mathbb{L}) \rtimes \mathbb{Z}$ the corresponding lampraag. Notice that:

\begin{fact}\label{fact:LampRaagPresentation}
The lampraag $A(\mathbb{L}) \rtimes \mathbb{Z}$ admits $\langle a,t \mid [a,tat^{-1}]=1 \rangle$ as a presentation. 
\end{fact}

\begin{proof}
It is clear that
$$\langle \ldots, a_{-1}, a_0, a_1, \ldots, t \mid ta_it^{-1}= a_{i+1} \text{ and } [a_i,a_{i+1}]=1 \text{ for every } i \in \mathbb{Z} \rangle$$
is a presentation of our lampraag. Since $a_i = t^ia_0t^{-i}$ for every $i \in \mathbb{Z}$, this infinite presentation can be simplified into the finite presentation given above. 
\end{proof}

\noindent
Interestingly, the lampraag over $\mathbb{Z}$ turns out to be connected to many other families of groups. For instance:
\begin{itemize}
	\item As a consequence of Fact~\ref{fact:LampRaagPresentation}, this is a one-relator group. One-relator groups have been extensively studied in combinatorial group theory.
	\item As already mention, this is an example of graph-wreath product \cite{MR3393469}, and more generally of halo product \cite{Halo}.
	\item As we will see in Section~\ref{section:Flip}, this is the fundamental group of a compact flip $3$-manifold with boundary.
	\item Its finite-index subgroup $A(\mathbb{L}) \rtimes 2\mathbb{Z}$, which admits $\langle a,b,t \mid [a,b]=[a,tbt^{-1}]=1 \rangle$ as a presentation, is a \emph{diagram group} \cite[Example~5.43]{MR4071367}. Interestingly, it is also proved in \cite{MR4071367} that $A(\mathbb{L}) \rtimes 2\mathbb{Z}$ is not a right-angled Artin group. 
	\item The lampraag $A(\mathbb{L}) \rtimes \mathbb{Z}$ is the fundamental group of a \emph{Cartesian graph of groups} \cite[Example~11.38]{QM}.
	\item It is not difficult to deduce from the description of centralisers in right-angled Artin groups (see Proposition~\ref{prop:CentralisersGP} below) that the lampraag $A(\mathbb{L}) \rtimes \mathbb{Z}$ cannot be a subgroup of a right-angled Artin group. However, its finite-index subgroup $A(\mathbb{L}) \rtimes 2 \mathbb{Z}$ is a simple subgroup of $A(P_3)$. Indeed, if we denote by $p,q,r,s$ the four vertices successively met along $P_3$, then the subgroup $\langle q,r,sp \rangle$ turns out to be isomorphic to $A(\mathbb{L}) \rtimes 2 \mathbb{Z}$.  
\end{itemize}

\noindent
The rest of the section is dedicated to the proof of the following statement, which shows that the lampraag $A(\mathbb{L}) \rtimes \mathbb{Z}$ can be thought of as a good model of $\mathrm{FB}_7$ up to commensurability.

\begin{prop}\label{prop:CommCartesianGG}
The group $\mathrm{FB}_7$ is commensurable to the lampraag $A(\mathbb{L}) \rtimes \mathbb{Z}$. 
\end{prop}

\noindent
Our argument allows us to be more explicit. To be precise, the proof of Proposition~\ref{prop:CommCartesianGG} shows that the map
$$\left\{ \begin{array}{ccc} \sigma_1\sigma_2 & \mapsto & a \\ \sigma_3\sigma_4 & \mapsto & t^2 \\ \sigma_5 \sigma_6 & \mapsto & tat^{-1} \end{array} \right.$$
induces an isomorphism from the index-$8$ normal subgroup $\langle \sigma_1\sigma_2, \sigma_3\sigma_4, \sigma_5\sigma_6 \rangle$ of $\mathrm{FB}_7$ to the index-$2$ subgroup $\langle a, tat^{-1}, t^2 \rangle = A(\mathbb{L}) \rtimes 2\mathbb{Z}$ of $A(\mathbb{L}) \rtimes \mathbb{Z}$, when presented as $\langle a,t \mid [a,tat^{-1}]=1 \rangle$ (see Fact~\ref{fact:LampRaagPresentation}). 

\begin{proof}[Proof of Proposition~\ref{prop:CommCartesianGG}.]
As a first step, let us consider the subgroup 
$$H:= \langle \sigma_1\sigma_2, \sigma_3, \sigma_4, \sigma_5 \sigma_6 \rangle$$ 
of $\mathrm{FB}_7$. One easily checks that this is a normal subgroup. Moreover, the quotient $\mathrm{FB}_7/H$ is the product of two cyclic groups of order $2$ generated by the images of $\sigma_1$ and $\sigma_5$. Therefore, $H$ is a normal subgroup of index $4$ in $\mathrm{FB}_7$ with $\{1,\sigma_1,\sigma_5, \sigma_1\sigma_5\}$ as a set of representatives. 

\medskip \noindent
Now, let us investigate the structure of $H$. For this purpose, consider the graph $\mathrm{M}:= \mathrm{Cayl}(\mathrm{FB}_7, \{\sigma_1, \ldots, \sigma_6 \})$. We know from Corollary~\ref{cor:RACGmedian} that this is a median graph. Moreover, it follows from Lemma~\ref{lem:AdjLabels} that two hyperplanes of $M$ labelled by $\sigma_3$ or $\sigma_4$ are never transverse; so the set $\mathcal{J}$ of all the hyperplanes of $\mathrm{M}$ labelled by $\sigma_3$ or $\sigma_4$ yields an arboreal structure on $\mathrm{M}$, i.e.\ the graph $T(\mathcal{J})$ whose vertices are the connected components of the graph $\mathrm{M}\backslash \backslash \mathcal{J}$ (obtained from $\mathrm{M}$ by removing the hyperplanes from $\mathcal{J}$) and whose edges connect two components whenever they are separated by a single hyperplane is a tree. Make $H$ act on $T(\mathcal{J})$. 

\medskip \noindent
Notice that there is a single $H$-orbit of vertices. Indeed, a vertex of $T(\mathcal{J})$ corresponds to a maximal subgraph of $\mathrm{M}$ all of whose edges are labelled by generators distinct from $\sigma_3$ and $\sigma_4$. In other words, the vertices of $T(\mathcal{J})$ correspond to the cosets of $K:=\langle \sigma_1, \sigma_2, \sigma_5, \sigma_6 \rangle$ in $\mathrm{FB}_7$. Since we saw that $1,\sigma_1,\sigma_5, \sigma_1\sigma_5$ are representatives modulo $H$, it suffices to verify that $K$, $\sigma_1K$, $\sigma_5K$, and $\sigma_1\sigma_5 K$ all lie in the same $H$-orbit. But this is clear since:
\begin{itemize}
	\item $\sigma_1K = \sigma_1\sigma_2 K$ with $\sigma_1\sigma_2 \in H$;
	\item $\sigma_5K = \sigma_5 \sigma_6 K$ with $\sigma_5 \sigma_6 \in H$;
	\item $\sigma_1 \sigma_5 K = \sigma_1 \sigma_2 \sigma_5 \sigma_6 K$ with $\sigma_1\sigma_2, \sigma_5 \sigma_6 \in H$.
\end{itemize}
The $H$-stabiliser of the vertex of $T(\mathcal{J})$ given by $K$ is $H \cap K = \langle \sigma_1\sigma_2, \sigma_5 \sigma_6 \rangle$. Next, notice that there are two $H$-orbits of edges in $T(\mathcal{J})$. Indeed, the edges of $T(\mathcal{J})$ correspond to the hyperplanes of $\mathrm{M}$ labelled by $\sigma_3$ or $\sigma_4$; or equivalently, according to Lemma~\ref{lem:QMhyp}, to the cosets of $E_1:= \langle \sigma_1, \sigma_3, \sigma_5,\sigma_6 \rangle$ and $E_2:= \langle \sigma_1,\sigma_2, \sigma_4, \sigma_6 \rangle$ in $\mathrm{FB}_7$. As previously, we deduce our claim from the following easy observations:
\begin{itemize}
	\item for every $\eta \in \{\sigma_1, \sigma_5, \sigma_1\sigma_5\}$, $\eta E_1 = E_1$ since $\eta \in E_1$;
	\item $\sigma_1 E_2 = E_2$ since $\sigma_1 \in E_2$;
	\item $\sigma_5 E_2 = \sigma_5 \sigma_6 E_2$ with $\sigma_5 \sigma_6 \in H$;
	\item $\sigma_1 \sigma_5 E_2  = \sigma_5 E_2 = \sigma_5 \sigma_6 E_2$ with $\sigma_5 \sigma_6 \in H$.
\end{itemize}
The $H$-stabiliser of the edge of $T(\mathcal{J})$ given by $E_1$ (resp.\ $E_2$) is $H \cap E_1 = \langle \sigma_3, \sigma_5\sigma_6 \rangle$ (resp.\ $H \cap E_2 = \langle \sigma_4, \sigma_1 \sigma_2 \rangle$). We conclude from these properties satisfied by the action of $H$ on $T(\mathcal{J})$ that $H$ decomposes as the following graph of groups:

\begin{center}
\includegraphics[width=0.7\linewidth]{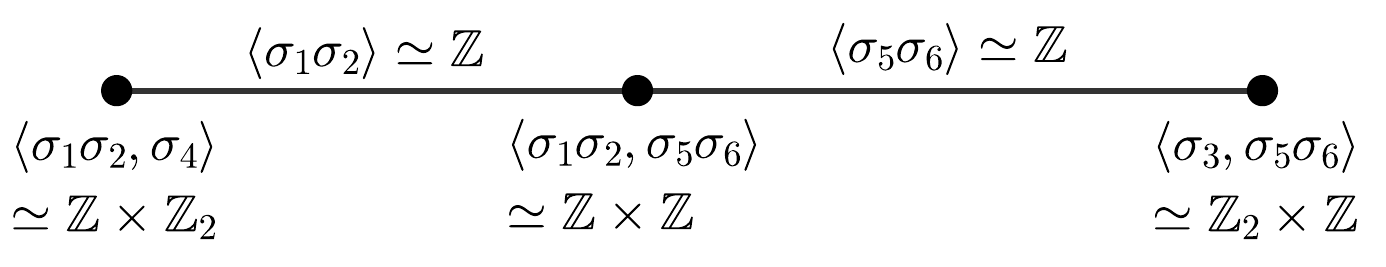}
\end{center}

\noindent
Alternatively, this amounts to describing $H$ as the following graph product:

\begin{center}
\includegraphics[width=0.7\linewidth]{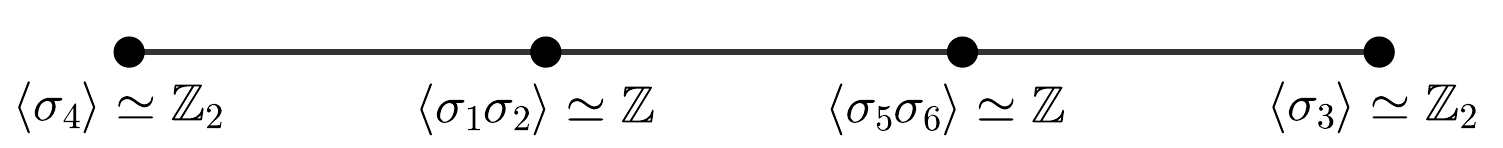}
\end{center}

\noindent
For convenience, we will write $\sigma_{12}:= \sigma_1\sigma_2$ and $\sigma_{56}:= \sigma_5 \sigma_6$. Consider the action of $H$ on the quasi-median graph $\mathrm{QM}$ given by this decomposition, as described in Section~\ref{section:QM}. Let $\mathcal{R}$ be the collection of all the hyperplanes of $\mathrm{QM}$ labelled by $\sigma_{12}$ or $\sigma_{56}$ and let $\Lambda$ denote the subgraph $\langle \sigma_3, \sigma_4 \rangle \subset \mathrm{QM}$. Notice that $\Lambda$ is a connected component of $\mathrm{QM} \backslash \backslash \mathcal{R}$. Of course, since $\langle \sigma_3,\sigma_4 \rangle$ is an infinite dihedral group, $\Lambda$ is just a bi-infinite line; namely 
$$\ldots, \ \sigma_4, \ \sigma_3 \sigma_4, \ \sigma_4\sigma_3, \ \sigma_4, \  1, \ \sigma_3, \ \sigma_3 \sigma_4, \ \sigma_3\sigma_4 \sigma_3, \ \ldots$$
The hyperplanes of $\mathcal{R}$ tangent to $\Lambda$ are the $\langle \sigma_3,\sigma_4\rangle$-translates of $J_{\sigma_{12}}$ and $J_{\sigma_{56}}$. The configuration in $\mathrm{QM}$ is the following:

\begin{center}
\includegraphics[width=\linewidth]{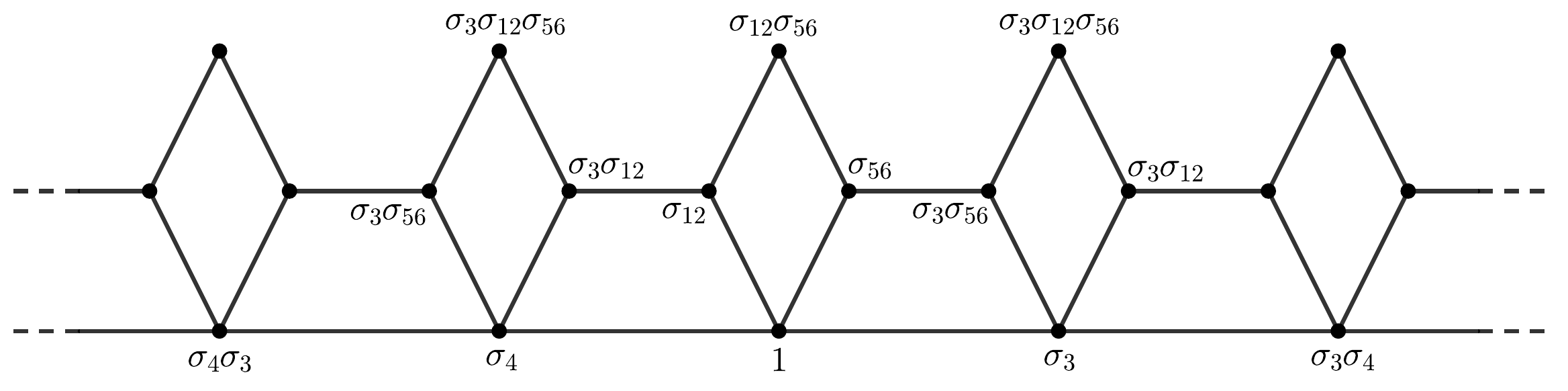}
\end{center}

\noindent
Thus, the crossing graph of the hyperplanes of $\mathcal{R}$ tangent to $\Lambda$ is also a bi-infinite line, and $\langle \sigma_3,\sigma_4 \rangle$ acts on it through two reflections fixing two adjacent vertices. Applying Proposition~\ref{prop:PingPong}, we obtain the decomposition:
$$H = \langle \text{conjugates of } \sigma_1\sigma_2 \text{ and } \sigma_5\sigma_6 \rangle \rtimes \langle \sigma_3,\sigma_4 \rangle \simeq A(\mathbb{L} ) \rtimes \mathbb{D}_\infty$$
where $\mathbb{D}_\infty$ acts on $\mathbb{Z}$ via two reflections (corresponding to $\sigma_3$ and $\sigma_4$) fixing $0$ and $1$ respectively. We conclude that the index-$2$ subgroup $\langle \sigma_1\sigma_2, \sigma_3\sigma_4, \sigma_5\sigma_6 \rangle$ of $H$, which is therefore an index-$8$ subgroup of $\mathrm{FB}_7$, is isomorphic to the index-$2$ subgroup $A(\mathbb{L}) \rtimes 2 \mathbb{Z}$ of the lampraag $A(\mathbb{L}) \rtimes \mathbb{Z}$. 
\end{proof}

\subsection{Some flip manifolds}\label{section:Flip}

\noindent
In this section, our goal is to describe $A(\mathbb{L}) \rtimes \mathbb{Z}$ and $A(P_4)$ as fundamental groups of two flip manifolds $M_1$ and $M_2$, and then to construct a third flip manifold $M_0$ that is a common finite-sheeted cover of $M_1$ and $M_2$. Since $A(\mathbb{L}) \rtimes \mathbb{Z}$ is commensurable to $\mathrm{PFB}_7$ according to Proposition~\ref{prop:CommCartesianGG}, this will allow us to deduce that $\mathrm{PFB}_7$ is commensurable to $A(P_4)$. Here, by a flip manifold, we mean a $3$-manifold (possibly with boundary) obtained by gluing copies of $\mathbb{S}^1 \times \mathbb{S}(b)$, where $\mathbb{S}(b)$ denotes a punctured sphere with $b \geq 1$ boundary components. Two copies of $\mathbb{S}^1 \times \mathbb{S}(b)$ will be always glued along a torus boundary component in such a way that meridians and longitudinals are switched.

\medskip \noindent
Let us begin by describing $A(\mathbb{L}) \rtimes \mathbb{Z}$ as the fundamental group of the flip manifold $M_1$ given by Figure~\ref{MOne}. The fundamental group of $\mathbb{S}^1 \times \mathbb{S}(3)$ can be identified with $\mathbb{Z} \times \mathbb{F}_2 = \langle a,b,c \mid [a,c]=[b,c]=1 \rangle$ where $a$ (resp.\ $b$) corresponds to the red (resp.\ blue) boundary component of $\mathbb{S}(3)$ and where $c$ is given by the $\mathbb{S}^1$-factor. The gluing identifies $c$ with $b$ and $a$ with $c$, so the fundamental group of $M_1$ admits
$$\langle a,b,c,t \mid [a,c]=[b,c]=1, \ tct^{-1}=b, \ tat^{-1}=c \rangle$$
as a presentation, which can be simplified as $\langle a,t \mid [a,tat^{-1}] =1 \rangle$. We conclude from Fact~\ref{fact:LampRaagPresentation} that $A(\mathbb{L}) \rtimes \mathbb{Z}$ is indeed isomorphic to the fundamental group of $M_1$. 
\begin{figure}
\begin{center}
\includegraphics[width=0.4\linewidth]{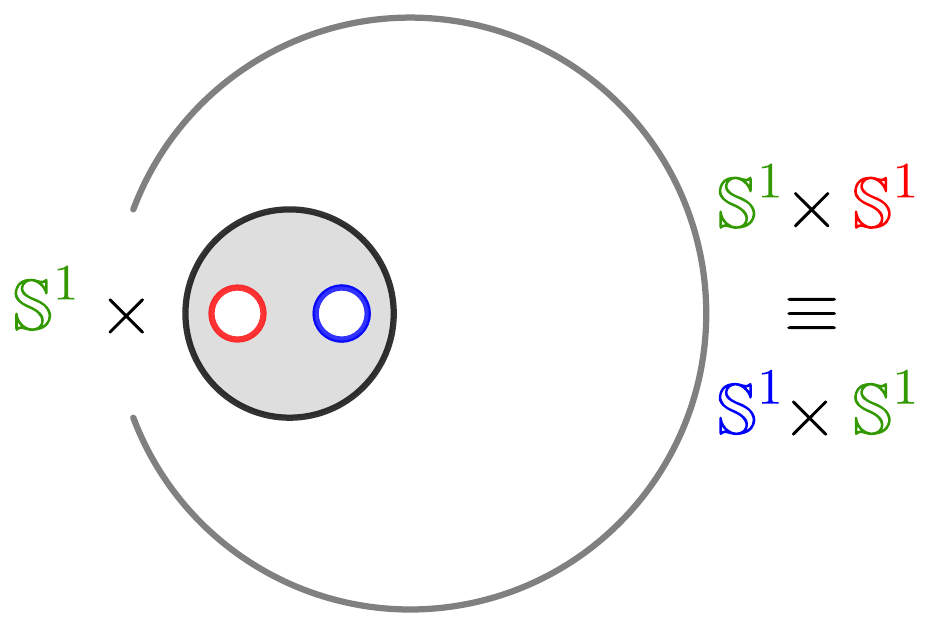}
\caption{The flip manifold $M_1$ whose fundamental group is $A(\mathbb{L}) \rtimes \mathbb{Z}$.}
\label{MOne}
\end{center}
\end{figure}

\medskip \noindent
Next, let us describe $A(P_4)$ as the fundamental group of the flip manifold $M_2$ given by Figure~\ref{MTwo}. The three copies of $\mathbb{S}^1 \times \mathbb{S}(3)$ have fundamental groups isomorphic to $\mathbb{Z} \times \mathbb{F}_2$, say respectively $\langle a_i, b_i,c_i \mid [a_i,c_i]=[b_i,c_i]=1 \rangle$ for $i=1,2,3$ where $a_i$ corresponds to the left (resp.\ right) inner boundary component of $\mathbb{S}(3)$ and where $c_i$ is given by the $\mathbb{S}^1$-factor. The first gluing identifies $c_1$ with $a_2$ and $b_1$ with $c_2$, and the second gluing identifies $c_2$ with $a_3$ and $b_2$ with $c_3$. Therefore, the fundamental group of $M_2$ admits
$$\left\langle \begin{array}{l} a_1,a_2,a_3,b_1, \\ b_2,b_3, c_1,c_2,c_3 \end{array}  \mid \begin{array}{l} [a_1,c_1]=[b_1,c_1]=[a_2,c_2]= [b_2,c_2]=[a_3,c_3]=[b_3,c_3] = 1, \\ c_1=a_2, \ b_1=c_2, \  c_2=a_3, \ b_2=c_3 \end{array} \right\rangle$$
as a presentation, which can be simplified as
$$\langle a_1, a_2,b_2,c_2, b_3  \mid [a_1, a_2]=[a_2,c_2]= [c_2,b_2]= [b_2,b_3]=1 \rangle.$$
This is clearly a presentation of $A(P_4)$, concluding the proof that $A(P_4)$ is indeed isomorphic to the fundamental group of $M_2$.
\begin{figure}
\begin{center}
\includegraphics[width=0.7\linewidth]{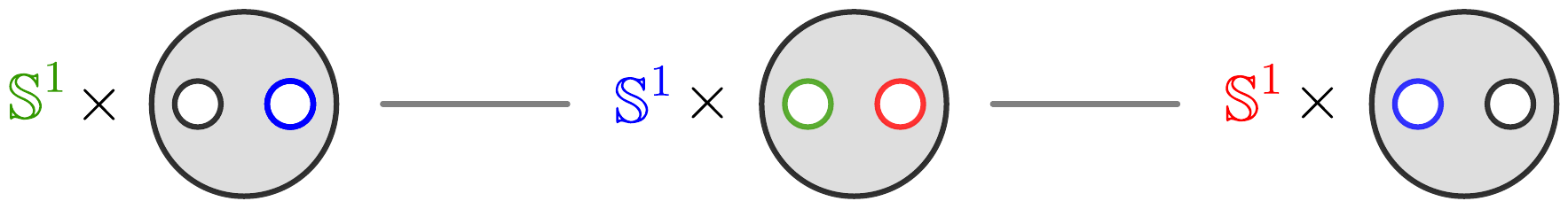}
\caption{The flip manifold $M_2$ whose fundamental group is $A(P_4)$.}
\label{MTwo}
\end{center}
\end{figure}
\begin{figure}
\begin{center}
\includegraphics[width=0.8\linewidth]{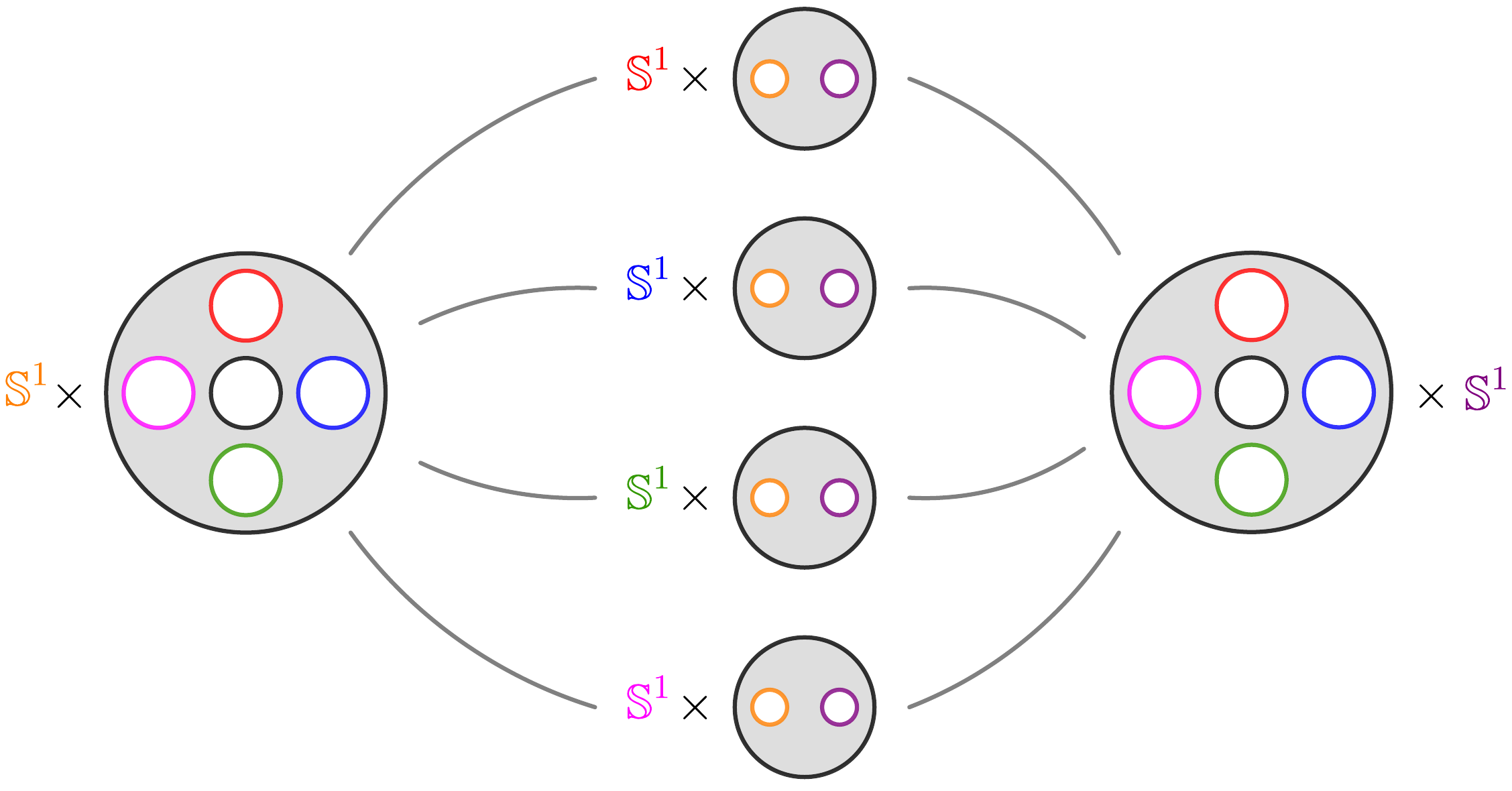}
\caption{The flip manifold $M_0$, a common cover of $M_1$ and $M_2$.}
\label{MZero}
\end{center}
\end{figure}

\medskip \noindent
Finally, let us construct a common finite cover of $M_1$ and $M_2$. Our flip manifold $M_0$ is described by Figure~\ref{MZero}. In order to describe the covering maps $M_0 \to M_1, M_2$, we need two specific covering maps $\alpha, \beta : \mathbb{S}(6) \to \mathbb{S}(3)$. They are respectively described by Figures~\ref{Alpha} and~\ref{Beta}. It is worth noticing that the restriction of $\alpha$ to each boundary component is a $2$-sheeted cover; and that the restriction of $\beta$ to a red (resp.\ blue) boundary component is a $4$-sheeted (resp.\ $1$-sheeted) cover.

\begin{figure}
\begin{center}
\includegraphics[trim = 0 26cm 20.5cm 0, clip, width=\linewidth]{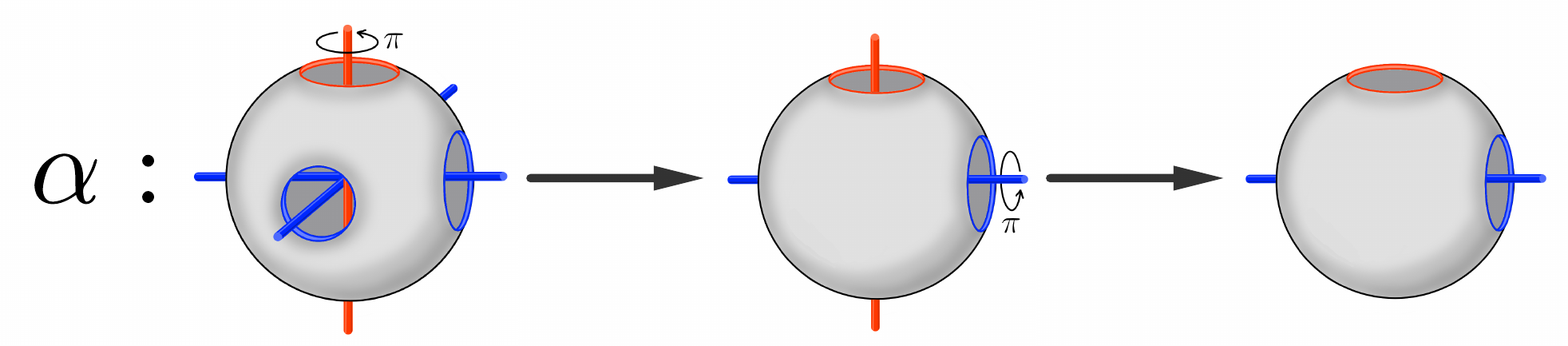}
\caption{The $4$-sheeted covering map $\alpha : \mathbb{S}(6) \to \mathbb{S}(3)$.}
\label{Alpha}
\end{center}
\end{figure}
\begin{figure}
\begin{center}
\includegraphics[trim = 0 26cm 36cm 0, clip, width=0.65\linewidth]{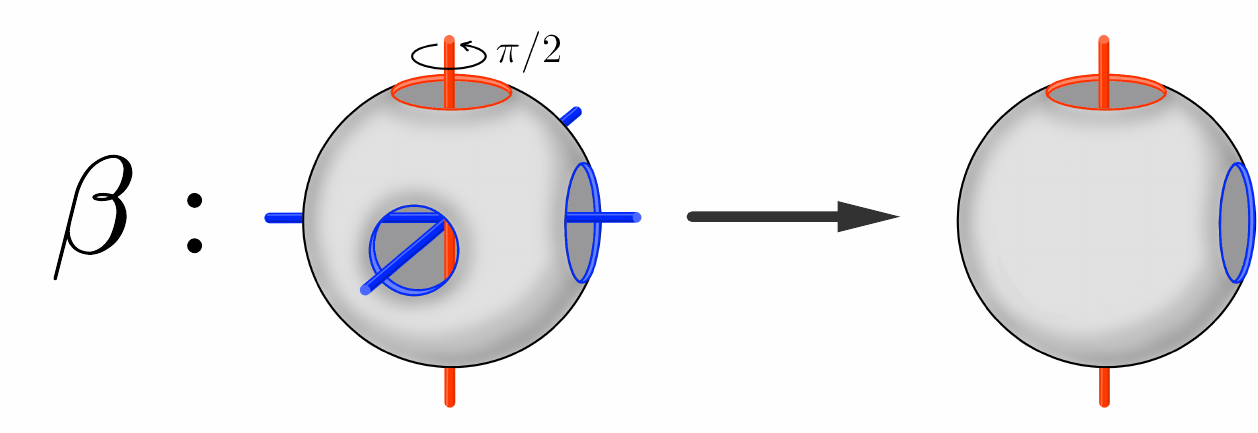}
\caption{The $4$-sheeted covering map $\beta : \mathbb{S}(6) \to \mathbb{S}(3)$.}
\label{Beta}
\end{center}
\end{figure}

\medskip \noindent
The covering map $\mu : M_0 \to M_1$ is defined as follows. The two pieces $\mathbb{S}^1 \times \mathbb{S}(6)$ of $M_0$ are sent to the piece $\mathbb{S}^1 \times \mathbb{S}(3)$ of $M_1$ through $(1, \alpha)$. And the four pieces $\mathbb{S}^1 \times \mathbb{S}(3)$ of $M_0$ are sent to the piece $\mathbb{S}^1 \times \mathbb{S}(3)$ of $M_1$ through $(2, \mathrm{id})$. By construction, these isolated maps are compatible with the gluings defining $M_0$, and we get a $4$-sheeted covering map $\mu : M_0 \to M_1$.

\medskip \noindent
The covering map $\nu : M_0 \to M_2$ is defined as follows. The two pieces $\mathbb{S}^1 \times \mathbb{S}(6)$ of $M_0$ are sent to the left and right pieces $\mathbb{S}^1 \times \mathbb{S}(3)$ of $M_2$ through $(1, \beta)$. And the four pieces $\mathbb{S}^1 \times \mathbb{S}(3)$ of $M_0$ are sent to the middle piece $\mathbb{S}^1 \times \mathbb{S}(3)$ of $M_2$ through $(4,\beta)$. By construction, these isolated maps are compatible with the gluings defining $M_0$, and we get a $4$-sheeted covering map $\nu : M_0 \to M_2$.

\medskip \noindent
Since we have constructed two $4$-sheeted covering maps $M_0 \to M_1, M_2$, it follows that the fundamental groups of $M_1$ and $M_2$ share an isomorphic $4$-index subgroups. Hence:

\begin{prop}
The lampraag $A(\mathbb{L}) \rtimes \mathbb{Z}$ and the right-angled Artin group $A(P_4)$ share an isomorphic $4$-index subgroup.
\end{prop}

\noindent
Combined with Proposition~\ref{prop:CommCartesianGG}, we conclude that Theorem~\ref{thm:Commensurable} holds.

\addcontentsline{toc}{section}{References}

\bibliographystyle{alpha}
{\footnotesize\bibliography{FlatBraidGroups}}

\Address

%

\end{document}